\documentclass{article}

\usepackage{amsmath} 
\usepackage{amssymb} 
\usepackage{mathtools} 
\usepackage{stmaryrd} 
\usepackage{centernot} 
\usepackage{amsthm} 
\usepackage{enumitem} 
\usepackage{mathrsfs} 
\usepackage[bb=boondox]{mathalfa} 
\usepackage{mathabx} 

\usepackage[hidelinks]{hyperref}
\usepackage{cleveref} 

\newtheorem{theorem}{Theorem}[section]
\newtheorem{lemma}[theorem]{Lemma}

\newtheorem{corollary}[theorem]{Corollary}
\newtheorem{proposition}[theorem]{Proposition}

\theoremstyle{definition}
\newtheorem{definition}[theorem]{Definition}

\theoremstyle{remark}
\newtheorem{remark}[theorem]{Remark}
\newtheorem {question}[theorem]{Question}

\numberwithin{equation}{section}

\newcommand{\NN}{\mathbb{N}}

\newcommand{\PP}{\mathbb{P}}
\newcommand{\lt}{{<}}
\newcommand{\gt}{{>}}

\newcommand{\dom}{\operatorname{dom}}

\newcommand{\uh}{\!{\upharpoonright}}
\newcommand{\inter}[1]{\llbracket #1 \rrbracket}

\newcommand{\C}{\mathcal{C}}
\newcommand{\T}{\mathcal{T}}

\newcommand{\R}{\mathcal{R}}
\newcommand{\M}{\mathcal{M}}
\newcommand{\Q}{\mathcal{Q}}
\newcommand{\F}{\mathcal{F}}
\renewcommand{\P}{\mathcal{P}}

\newcommand{\X}{\mathcal{X}}

\renewcommand{\L}{\mathscr{L}}

\newcommand{\kP}{\mathfrak{P}}

\newcommand{\RT}{\mathsf{RT}}
\newcommand{\SRT}{\mathsf{SRT}}

\newcommand{\WKL}{\mathsf{WKL}}
\newcommand{\ACA}{\mathsf{ACA}}
\newcommand{\RCA}{\mathsf{RCA}}

\newcommand{\COH}{\mathsf{COH}}

\newcommand{\CC}{\mathsf{CC}}
\newcommand{\KL}{\mathsf{KL}}

\newcommand{\Qsf}{\mathsf{Q}}
\newcommand{\Psf}{\mathsf{P}}

\usepackage{xcolor} 

\newcommand{\ie}{i.e.\ }
\newcommand{\ce}{c.e.\ }

\def\qt#1{``#1''}%

\title{Cross-constraint basis theorems\\ and products of partitions}
\author{Julien Cervelle \and William Gaudelier \and Ludovic Levy Patey}
\date{\today}

\begin{document}

\setitemize{itemsep=0pt, topsep=1em}
\setenumerate{itemsep=0pt, topsep=1em}

\maketitle

\begin{abstract}
We both survey and extend a new technique from Lu Liu to prove separation theorems between products of Ramsey-type theorems
over computable reducibility. We use this technique to show that Ramsey's theorem for $n$-tuples and three colors is not computably reducible to finite products of Ramsey's theorem for $n$-tuples and two colors.
\end{abstract}


\section{Introduction}

In this article, we consider mathematical theorems as problems~$\Psf$, formulated in terms of \emph{instances} and \emph{solutions}.
For example, K\"onig's lemma states that every infinite, finitely branching tree admits an infinite path. Here, an instance of $\KL$ is an infinite, finitely branching tree $T \subseteq \NN^{<\NN}$, and a solution to~$T$ is an infinite path~$P \in [T]$.

There are many ways to compare the strength of mathematical problems. The most well-known approach is proof-theoretic, using \emph{reverse mathematics}.
It uses subsystems of second-order arithmetic, with a base theory, $\RCA_0$, capturing \emph{computable mathematics}. Then, if~$\RCA_0 \vdash \Qsf \rightarrow \Psf$, then $\Qsf$ is at least as strong as $\Psf$, in the sense that one can solve~$\Psf$ using multiple applications of $\Qsf$ as a black box, with only computable manipulations. There exist other approaches, more computability-theoretic, such as the \emph{Weihrauch reduction} and \emph{computable reduction}. A problem~$\Psf$ is \emph{computably reducible} to~$\Qsf$ (written $\Psf \leq_c \Qsf$) if for every $\Psf$-instance~$X$, there exists an $X$-computable $\Qsf$-instance~$\widehat{X}$ such that, for every $\Qsf$-solution~$\widehat{Y}$ of~$\widehat{X}$, $X \oplus \widehat{Y}$ computes a $\Psf$-solution to~$X$. Weihrauch reduction is a uniform variant of computable reduction.

The proof-theoretic and computability-theoretic approaches are related, in that when restricting reverse mathematics to \emph{$\omega$-models}, that is, models whose first-order part consists of the standard integers together with the usual operations, then an implication over~$\RCA_0$ is a generalization of computable reduction, in which multiple successive applications are allowed. There are similar links between intuitionistic reverse mathematics and the Weihrauch reduction. From this viewpoint, computable reduction is \qt{ressource-sensitive}, in that only one application of~$\Qsf$ is allowed to solve~$\Psf$. Each approach has its own interest, as it reveals a different aspect of the relation between $\Psf$ and~$\Qsf$.

\subsection{Ramsey's theorem}

In this article, we consider one particular family of problems, based on \emph{Ramsey's theorem}. Given a set~$X \subseteq \NN$, we write $[X]^n$ for the set of all $n$-element subsets of~$X$. Given a coloring $f : [\NN]^n \to k$, a set~$H \subseteq \NN$ is \emph{$f$-homogeneous} if \(f\) is constant on~$[H]^n$. Ramsey's theorem for $n$-tuples and $k$-colors ($\RT^n_k$) states the existence, for every coloring $f : [\NN]^n \to k$, of an infinite $f$-homogeneous set. The statement~$\RT^n_k$ has two parameters.

The computational power of~$\RT^n_k$ depending on~$n$ is well-understood for computable instances. For $n = 1$, there is always a computable solution. For $n \geq 2$, Jockusch~\cite[Theorem 5.1]{jockusch1972ramsey} proved the existence of a computable instance of~$\RT^n_2$ with no $\Sigma^0_n$ solution. He also showed \cite[Theorem 5.5]{jockusch1972ramsey} that every computable instance of~$\RT^n_k$ admits a $\Pi^0_n$ solution. Cholak, Jockusch, and Slaman~\cite[Theorem 12.1]{cholak_jockusch_slaman_2001} proved that for every computable instance of~$\RT^n_k$, every PA degree over~$\emptyset^{(n-2)}$ computes the jump of a solution. Seetapun~\cite[Theorem 2.1]{seetapun1995strength} proved that for every computable instance of~$\RT^n_k$ and for every non~$\Delta^0_{n-1}$ set~$C$, there exists a solution that does not compute~$C$. Liu~\cite[Theorem 1.5]{liu2012rt22} proved that every computable instance of~$\RT^2_k$ admits a solution of non-PA degree. On the other hand, for~$n \geq 3$, Hirschfeldt and Jockusch~\cite[Corollary 2.2]{hirschfeldt2016notions} proved the existence of a computable instance of~$\RT^n_2$ such that every solution is of PA degree over~$\emptyset^{(n-2)}$.

Translated in terms of reverse mathematics, $\RT^1_2$ is provable over~$\RCA_0$, $\RT^2_2$ is strictly in between~$\RCA_0$ and $\ACA_0$, and incomparable with~$\WKL_0$, and $\RT^n_2$ is equivalent to~$\ACA_0$ over~$\RCA_0$ for every~$n \geq 3$. See Simpson~\cite{simpson_2009} for a presentation of~$\RCA_0$, $\WKL_0$ and~$\ACA_0$. The number~$k$ of colors is not relevant when it is standard, using a color-blindness argument. Indeed, given an instance of~$\RT^n_{k^2}$, one can create an instance of~$\RT^n_k$ by grouping the colors into blocks of size~$k$, and, given a solution, define another instance of~$\RT^n_k$ by working on the domain of the first solution. For this reason, the power of~$\RT^n_k$ depending on~$k$ was not studied until recently.

The color blindness argument shows that~$\RCA_0 \vdash \RT^n_k \rightarrow \RT^n_{k^2}$ by two successive applications of~$\RT^n_k$. This kind of argument does not hold for a resource-sensitive notion such as computable reduction. Patey~\cite[Corollary 3.14]{patey2016weakness} proved that for every~$n \geq 2$ and every~$k$, $\RT^n_{k+1} \not \leq_c \RT^n_k$.

\subsection{Products of problems}

There exist various operators on mathematical problems, coming essentially from the study of Weihrauch degrees. Among these, we shall consider two operators:
\begin{itemize}
    \item The \emph{star} of a problem~$\Psf$ is the problem $\Psf^*$ whose instances are finite tuples $(X_0, \dots, X_{n-1})$ of instances  of~$\Psf$ for some~$n \in \NN$ and whose solutions are finite tuples $(Y_0, \dots, Y_{n-1})$ such that $Y_i$ is a $\Psf$-solution to~$X_i$  for every~$i < n$.
    \item The \emph{parallelization} of a problem~$\Psf$ is the problem $\widehat{\Psf}$ whose instances are infinite sequences of~$\Psf$-instances~$X_0, X_1, \dots$ and whose solutions are infinite sequences~$Y_0, Y_1, \dots$ such that for every~$n \in \NN$, $Y_n$ is a $\Psf$-solution to~$X_n$.
\end{itemize}
When considering a reduction~$\Psf \leq_c \Qsf^*$, one is allowed to use an arbitrarily large, but finite number~$n$ of instances of~$\Qsf$ to solve an instance~$X$ of~$\Psf$, where $n$ depends on~$X$. However, the instances of~$\Qsf$ must be simultaneously chosen, in that they are not allowed to depend on each others' solutions. This simultaneity prevents from using the standard color blindness argument, and motivates the following question:

\begin{question}\label[question]{quest:rtnk-prod-red}
Given~$n, k \geq 1$, is $\RT^n_{k+1} \leq_c (\RT^n_k)^*$?
\end{question}

The case $n = 1$ is not interesting since~$\RT^n_k$ is computably true, but was studied in the context of Weihrauch degrees by Dorais et al~\cite{dorais2016uniform}, Hirschfeldt and Jockusch~\cite{hirschfeldt2016notions} and Dzhafarov and al~\cite{dhzafarov2020ramsey}. The first interesting case for computable reduction is then~$n = 2$. Let us first illustrate why the technique used to prove that $\RT^2_{k+1} \not \leq_c \RT^2_k$ fails when considering products. Patey~\cite{patey2016weakness} used an analysis based on preservation of hyperimmunities. A function~$g : \NN \to \NN$ is \emph{hyperimmune} if it is not dominated by any computable function. An infinite set~$A$ is hyperimmune if its principal function is hyperimmune, where the \emph{principal function} $p_A$ of a set~$A\coloneqq\{ x_0 < x_1 < \ldots \}$ is the function $n \mapsto x_n$. Patey proved \cite[Theorem 3.11]{patey2016weakness} that for every~$k \geq 1$, every $(k+1)$-tuple of hyperimmune functions~$g_0, \dots, g_k$, and every computable instance~$X$ of~$\RT^n_k$, there exists a solution~$Y$ such that at least two among the hyperimmune functions remain $Y$-hyperimmune, but that it is not the case for computable instances of~$\RT^n_{k+1}$. This property fails when considering the star operator, as the following lemma shows:

\begin{lemma}[{Cholak et al~\cite{cholak2020someresults}}]\label[lemma]{lem:product-hyperimmunity}
There exist 4 hyperimmune functions $g_0, \dots, g_3$ and two computable colorings $f_0, f_1 : [\NN]^2 \to 2$ such that for every infinite $f_0$-homogeneous set~$H_0$ and every infinite $f_1$-homogeneous set~$H_1$, at most one $g_i$ is $H_0 \oplus H_1$-homogeneous.
\end{lemma}
\begin{proof}
Indeed, consider a $\Delta^0_2$ 4-partition $A_0 \sqcup \dots \sqcup A_3 = \NN$ such that for every~$i < 4$, $\overline{A}_i$ is hyperimmune. 
For every~$i < 4$, $g_i$ is the principal function of~$\overline{A}_i$.
Let $f_0, f_1$ be computable instances of~$\RT^2_2$ such that for every~$x$, $\lim_y f_0(x, y) = 1$ iff $x \in A_0 \cup A_1$ and $\lim_y f_1(x, y) = 1$ iff $x \in A_0 \cup A_2$. Every infinite $f_0$-homogeneous set~$H_0$ is either included in~$A_0 \cup A_1$ or in $A_2 \cup A_3$, and every infinite $f_1$-homogeneous set~$H_1$ is either included in~$A_0 \cup A_2$ or in~$A_1 \cup A_3$.

Note that if a set~$H \subseteq A_i \cup A_j$ for some~$i < j < 4$, then $p_H$ dominates each $g_k$ for~$k \in \{0,1,2,3\}-\{i,j\}$,
in which case none of those $g_k$ are $H$-hyperimmune. Thus, either $g_0$ and $g_1$ or $g_2$ and $g_3$ are not~$H_0$-hyperimmune, and either~$g_0$ and $g_2$, or $g_1$ and $g_3$ are not~$H_1$-hyperimmune. It follows that at most one of the $g_i$ is $H_0 \oplus H_1$-hyperimmune.
\end{proof}

The instances of~$\RT^2_2$ created in \Cref{lem:product-hyperimmunity} are called \emph{stable}, that is, for every~$x$, $\lim_y f(x, y)$ exists. Let~$\SRT^2_k$ denote the restriction of~$\RT^2_k$ to stable instances.
Liu~\cite{liu2023coding} proved that $\SRT^2_3 \not \leq_c (\SRT^2_2)^*$, answering a question by Cholak et al~\cite{cholak2020someresults}.
His proof involved completely new combinatorics, which will be presented in this article. We will also extend his result to give a complete answer to \Cref{quest:rtnk-prod-red} and prove that $\RT^n_3 \not \leq_c (\RT^n_2)^*$ for every~$n \geq 2$.
Before presenting Liu's approach, note that the reduction holds when considering parallelization, but for a completely different reason. 

\begin{lemma}
For every~$n, k \geq 1$, $\RT^{n+1}_k \leq_c \widehat{\RT^n_2}$.
\end{lemma}
\begin{proof}
By Brattka and Rakotoniaina~\cite[Corollary 3.30]{brattka2017uniform}, $\WKL^{(n)} \leq_c \widehat{\RT^n_2}$, where $\WKL^{(n)}$ is the problem whose instances are $\Delta^0_{n+1}$ approximations of infinite binary trees, and whose solutions are infinite paths through the trees. It follows that for every set~$X$, there is an $X$-computable instance of~$\widehat{\RT^n_2}$ such that every solution is of PA degree over~$X^{(n)}$. By Cholak, Jockusch and Slaman~\cite{cholak_jockusch_slaman_2001}, for every instance~$X$ of~$\RT^{n+1}_k$, every PA degree over~$X^{(n)}$ computes the jump of a solution to~$X$.
\end{proof}

\subsection{Definitions and notation}

We assume familiarity with computability theory (see Soare~\cite{soare2016turing}) and reverse mathematics (see any of Simpson~\cite{simpson_2009}, Hirschfeldt~\cite{hirschfeldt2017slicing} or Dzhafarov and Mummert~\cite{dzhafarov2022reverse}).

We identify an integer~$k \in \NN$ with the set $\{0, \dots, k-1\}$. The set $[X]^n$ of $n$-element subsets of~$X$ is in one-to-one correspondence with the set of increasing ordered $n$-tuples over~$X$.
Thus, we shall write for example $f : [\NN]^n \to k$ for $f : [\NN]^n \to \{0, \dots, k-1\}$, and $f(x_0, \dots, x_{n-1})$ for $f(\{x_0, \dots, x_{n-1}\})$, assuming $x_0 < \dots < x_{n-1}$.

Given~$k \in \NN$, let $k^\NN$ be the set of all infinite $k$-valued sequences, and $k^{<\NN}$ be the set of all finite $k$-valued strings.
We use Greek letters $\sigma, \tau, \mu, \rho, \dots$ to denote finitary strings, capital Latin letters $X, Y, Z$ to denote finite or infinite sets of integers, or infinite $k$-valued sequences. Given $k, n \in \NN$, $k^{=n}$ denotes the set of $k$-valued strings of length exactly~$n$. We define $k^{<n}$ and $k^{\leq n}$ accordingly.

Given two strings~$\sigma, \tau$, we let $|\sigma|$ denote the length of~$\sigma$, and $\sigma \preccurlyeq \tau$ means that $\sigma$ is a (non-strict) prefix of~$\tau$. We also write $\sigma \prec X$ to mean that $\sigma$ is an initial segment of~$X$. We let $[\sigma] \coloneqq \{ X \in k^\NN : \sigma \prec X \}$. The choice of~$k$ will be clear from the context. The letter~$\varepsilon$ denotes the empty string.

In the whole document, we fix \(r\in\NN\), and use the following notations. For $N \in \NN$ and $u \in \{\NN, {<}\NN, {\leqslant}N, {=}N \}$, we let $\X_u(0) \coloneqq 3^u$, $\X_u(1) \coloneqq (2^u)^r$ and $\X_u \coloneqq \X_u(0) \times \X_u(1)$. For simplicity, when $u$ is omitted it means $u \coloneqq \NN$, \ie $\X\coloneqq3^\NN\times(2^\NN)^r$.

\subsection{Organization of the paper}

We start by giving a general overview of the construction in~\Cref{sect:core-ideas}, by considering separations of theorems in general, then specializing to Ramsey-type theorems, and then to the actual separation of this article.
Then, in \Cref{sect:general-framework}, we introduce the main concepts, and prove a parameterized version of the theorems of Liu's article, justifying the study of cross-constraint basis theorems. In \Cref{sec:basis}, we prove several cross-constraint basis theorems. More precisely, we reprove the $\Delta^0_2$ basis theorem as a warm-up, give a new proof of the cone avoidance basis theorem, and prove a preservation of non-$\Sigma^0_1$ definitions and a low basis theorem. In \Cref{sect:gamma-hyperimmunity}, study the notion of $\Gamma$-hyperimmunity introduced by Liu, prove that $\COH$ preserves this notion, and deduce from it our main theorem, that is, $\RT^n_3 \not\leq_c (\RT^n_2)^*$ for~$n \geq 2$.

\section{Core ideas}\label[section]{sect:core-ideas}

In this section, we give the general picture of a proof for separating a theorem from another over computable reducibility.
Then, we specialize the idea to the particular question of $\RT^2_{k+1} \leq_c (\RT^2_k)^*$ and explain the core ideas of Liu's technique. Note that the terminology has been freely altered from Liu's original article.

\subsection{Separating theorems}

Given two problems $\Psf$ and $\Qsf$, in order to prove that $\Psf \not \leq_c \Qsf$, one needs to construct an instance $X_{\Psf}$ of~$\Psf$ whose solutions are difficult to compute, and, for every $X_{\Psf}$-computable instance~$X_{\Qsf}$ of~$\Qsf$, a solution~$Y_{\Qsf}$ to~$X_{\Qsf}$ that does not $X_{\Psf}$-compute any solution to~$X_{\Psf}$. The framework was used in its whole generality by Lerman, Solomon, and Towsner~\cite{lerman2013separating}, but all the currently known separations over reverse mathematics or over computable reduction can be done by constructing a computable instance of~$\Psf$.

The construction of the instance of~$\Psf$ can be done either by a priority construction or by an effectivization of a forcing construction. In many cases, it is constructed using the finite extension method, that is, an effectivization of Cohen forcing. This will be the case in our article (see \Cref{exists-gamma-hyp-coloring}).

Given a computable instance~$X_{\Qsf}$ of~$\Qsf$, the solution $Y_{\Qsf}$ is usually built by forcing, in a forcing notion $(\PP_{X_{\Qsf}}, \leq)$.
The solution has to satisfy two types of properties:
\begin{itemize}
    \item \emph{Structural properties}: being a solution to~$X_{\Qsf}$. These properties are generally ensured by the very definition of the notion of forcing.
    \item \emph{Computational properties}: not computing a solution to~$X_{\Psf}$. These properties are divided into countably many requirements, by considering each Turing functional individually. Given a requirement~$\R_e$, one must prove that the set of conditions forcing~$\R_e$ is dense.
\end{itemize}
There is often a tension between the structural properties which provide some computational strength, and the computational properties which require some weakness. There is however some degree of freedom in the computational properties, as they are parameterized by the instance~$X_{\Psf}$ of~$\Psf$ on which we have the hand.

The idea, coming from Lerman, Solomon and Towsner~\cite{lerman2013separating}, consists in building the instance $X_{\Psf}$ considering each tuple $(X_{\Qsf}, c, \R_e)$ at a time, where $X_{\Qsf}$ is a computable instance of~$\Qsf$, $c$ is a forcing condition in $(\PP_{X_{\Qsf}}, \leq)$, and $\R_e$ is a requirement. Given a partial approximation of~$X_{\Psf}$ and a tuple $(X_{\Qsf}, c, \R_e)$, ask whether there is an extension $d \leq c$ forcing $\Phi_e^{Y_{\Qsf}}$ to output enough bits of information. If so, complete $X_{\Psf}$ so that it diagonalizes against the functional. Otherwise, there is an extension $d \leq c$ forcing $\Phi_e^{Y_{\Qsf}}$ to be partial. The counter-intuitive part of this approach is that the satisfaction of the requirements is ensured by the construction of~$X_{\Psf}$ instead of the construction of the solutions~$Y_{\Qsf}$ to computable instances of~$\Qsf$.

As explained by Patey~\cite{patey2017iterative}, one can polish the previous construction, by abstracting the construction steps of $X_{\Qsf}$ to consider every operation with the same definitional properties, yielding some kind of genericity property. For example, the separation of the Erd\H{o}s-Moser theorem and the Ascending Descending Sequence from Ramsey's theorem for pairs~\cite{lerman2013separating} were polished in~\cite{patey2017iterative} and \cite{patey2018partial} to obtain hyperimmunity and dependent hyperimmunity, respectively. In this article, the polishing step yields $\Gamma$-hyperimmunity (see \Cref{def:gamma-hyperimmunity}).

\subsection{Separating Ramsey-type theorems}

In the particular case of Ramsey-type theorems, there exists a well-established sub-scheme of construction. Many Ramsey-type theorems are of the form \qt{For every $k$-coloring of the $n$-tuples of an infinite structure, there exists an infinite isomorphic substructure over which all the $n$-tuples satisfy some properties}. In the case of Ramsey's theorem, the infinite structure is $(\NN, <)$, and the property is homogeneity, but one can consider weaker properties, such as transitivity, in which case one obtains the Erd\H{o}s-Moser theorem. One can also consider tree structures, yielding the tree theorem~\cite{chubb2009reverse} or Milliken's tree theorem~\cite{angles2024milliken}. These theorems are usually proven by induction, by constructing so-called pre-homogeneous substructures. In the case of Ramsey's theorem, an infinite set~$X$ is \emph{pre-homogeneous} for a coloring $f : [\NN]^{n+1} \to k$ if for every $\vec{x} \in [X]^n$ and every larger~$y_0, y_1 \in X$, $f(\vec{x}, y_0) = f(\vec{x}, y_1)$. 

Although pre-homogeneity is the natural notion to consider from a combinatorial viewpoint, the computability-theoretic practice has shown the interest of a weaker notion of \qt{delayed pre-homogeneity} called cohesiveness. Each Ramsey-type theorem has its own notion of cohesiveness. In the case of Ramsey's theorem, this yields the following definition.

\begin{definition}
An infinite set~$C$ is \emph{cohesive} for an infinite sequence of sets~$R_0, R_1, \dots$ if for every~$n$, $C \subseteq^{*} R_n$ or $C \subseteq^{*} \overline{R}_n$, where~$\subseteq^{*}$ means \qt{included up to finite changes}. $\COH$ is the statement \qt{Every infinite sequence of sets admits an infinite cohesive set.}
\end{definition}

Given a coloring~$f : [\NN]^{n+1} \to k$, one can consider the sequence of sets $\vec{R}\coloneqq\langle R_{\vec{x}, z} : \vec{x} \in [\NN]^n, z < k \rangle$ defined by $R_{\vec{x}, z}\coloneqq \{ z \in \NN : f(\vec{x},z) = y \}$. Given an infinite $\vec{R}$-cohesive set~$C$, the coloring $f$ restricted to $[C]^{n+1}$ is \emph{stable}, that is, for every $\vec{x} \in [\NN]^n$, $\lim_{y \in C} f(\vec{x}, y)$ exists. This induces a $\Delta^0_2(C)$ coloring $\widehat{f} : [C]^n \to k$. Cohesiveness is therefore a bridge between computable instances of~$\RT^{n+1}_k$ and $\Delta^0_2$ instances of~$\RT^n_k$.

Cohesiveness has almost no computational power. Indeed, by delaying pre-homogeneity, the statement becomes about the jump of sets. More precisely, $\COH$ is computably equivalent to the statement \qt{For every $\Delta^0_2$ infinite binary tree, there exists a $\Delta^0_2$ path} (see Belanger~\cite{Blanger2022ConservationTF}). Most of the properties used to separate Ramsey-type statements are preserved by~$\COH$. This phenomenon can be explained by the fact that every set can be made $\Delta^0_2$ without affecting too much the ground model (see Towsner~\cite{towsner2015maximum}). This will be again the case in our article (\Cref{thm:coh-preserves-gamma}). Because of this, the question of the separation $\RT^2_{k+1} \not\leq_c (\RT^2_k)^*$ becomes a question about separating $\Delta^0_2$ instances of~$\RT^1_{k+1}$ from finite products of $\Delta^0_2$ instances of~$\RT^1_k$.

More generally, given two Ramsey-type statements $\Psf^n$, $\Qsf^n$ parameterized by the dimension of the $n$-tuples, the question of $\Psf^{n+1} \leq_c \Qsf^{n+1}$ is often reduced to the corresponding question with $\Delta^0_2$ instances of $\Psf^n$ and $\Qsf^n$. The experience shows that almost all the known separations consist in actually constructing a $\Delta^0_2$ instance of~$\Psf^n$ which defeats not only all the $\Delta^0_2$ instances of~$\Qsf^n$, but \emph{all} the instances of $\Qsf^n$, with no effectiveness restriction (see \cite{patey2016weakness,patey2016strength} for examples). The previous remark about Towsner's work shows that this apparently stronger diagonalization is often without loss of generality. This will be again the case in our article, and the $\Delta^0_2$ instance of~$\RT^1_{k+1}$ will defeat all the finite products of instances of~$\RT^1_k$ (\Cref{param-1}).

Building a single instance of~$\RT^1_k$ which defeats simultaneously uncountably many instances of~$(\RT^1_{k+1})^*$ raises new difficulties, as the sequence of all tuples $(X_{\Qsf}, c, \R_e)$ is not countable anymore. Thankfully, we shall see that there exists a single countable notion of forcing $(\PP, \leq)$ such that $\PP_X \subseteq \PP$ for every $\Qsf$-instance~$X$. Moreover, given a condition $c \in \PP$, the class $I(c)$ of all $\Qsf$-instances~$X$ such that $c \in \PP_X$ is a compact class. One will exploit this compactness to defeat all $\Qsf$-instances~$X \in I(c)$ simultaneously.

\subsection{Cross-constraint techniques}

The setting is therefore the following: in order to prove that $\Psf^{n+1} \not \leq_c \Qsf^{n+1}$, one builds a $\Delta^0_2$ instance $X$ of~$\Psf^n$ such that for every instance $\widetilde{X}$ of~$\Qsf^n$, there is a $\Qsf^n$-solution~$\widetilde{Y}$ to~$\widetilde{X}$ that does not compute any $\Psf^n$-solution to~$X$. 

The instance~$X$ of~$\Qsf^n$ is built by an effectivization of Cohen forcing. For example, in this article, to prove that $\RT^2_{k+1} \not \leq_c (\RT^2_k)^*$, we will build a $\Delta^0_2$ instance $f$ of~$\RT^1_{k+1}$ using an increasing sequence of $(k+1)$-valued strings $\sigma_0 \prec \sigma_1 \prec \dots$ and let $f$ be the limit of this sequence.

Let $(\PP, \leq)$ be a countable notion of forcing used to build solutions to every instance of~$\Qsf^n$. At stage~$s$, assuming the Cohen condition $\sigma_s$ has been defined, consider the next pair~$(c, \R_e)$ where $c \in \PP$ and $\R_e$ is a requirement saying that $\Phi_e^G$ is not a solution to the $\Psf^n$-instance. Consider the class~$\C \subseteq \dom \Psf^n \times I(c)$ of all pairs $(X, \widetilde{X})$ such that there is an extension~$d \leq c$ with $\widetilde{X} \in I(c)$ forcing $\Phi_e^G$ to be partial or a $\Psf^n$-solution to $X$. There are two cases:
\begin{itemize}
    \item Case 1: the class~$\C$ is \emph{left-full} below~$\sigma_s$, that is, for every instance~$X$ of~$\Psf^n$ extending~$\sigma_s$, there exists a $\Qsf^n$-instance~$\widetilde{X}$ such that $(X, \widetilde{X}) \in \C$. Then, by some appropriate basis theorem which depends on the combinatorics of~$\Psf^n$ and $\Qsf^n$, there exist multiple pairs $(X_0, \widetilde{X}_0), \dots, (X_{k-1}, \widetilde{X}_{k-1})$ in~$\C$ such that $X_0, \dots, X_{k-1}$ are \emph{incompatible}, in the sense that there is no set which is a solution to all these $\Psf^n$-instances simultaneously, while $\widetilde{X}_0, \dots, \widetilde{X}_{k-1}$ are compatible as $\Qsf^n$-instances. Then, by building a solution to the $\Qsf^n$-instance which will be simultaneously a solution to $\widetilde{X}_0, \dots, \widetilde{X}_{k-1}$, this forces $\Phi_e^G$ to be partial, hence to satisfy~$\R_e$.
    \item Case 2: the class~$\C$ is not left-full below~$\sigma_s$. Then, there exists a $\Psf^n$-instance~$X$ extending~$\sigma_s$ such that, for every~$\Qsf^n$-instance~$\widetilde{X} \in I(c)$, $c$ forces $\Phi_e^G$ not to be a $\Psf^n$-solution to~$X$. By compactness of~$I(c)$, an initial segment $\sigma_{s+1} \prec X$ is sufficient to witness this diagonalization, hence to satisfy~$\R_e$.
\end{itemize}

The general idea of cross-constraint techniques takes its roots in Liu's proof of separation of Ramsey's theorem for pairs from weak K\"onig's lemma~\cite{liu2012rt22}, in a slightly different setting. Indeed, $\Psf^{n+1}$ was $\WKL$, which is known to admit a maximally difficult instance, so only~$\Qsf^n$ was built. In that article, he considered the class $\C$ of all pairs $(f, \widetilde{X})$ such that $f$ is a partial function with finite support, and $\widetilde{X} \in I(c)$ is an instance of~$\RT^1_2$.

\section{General framework}\label[section]{sect:general-framework}

In this section, we define the fundamental notions of left-full cross-tree, and prove the main theorems parameterized by the cross-constraint ideals. The basis theorems proven in \Cref{sec:basis} will show the existence of cross-constraint ideals with various computability-theoretic properties, and will be used to answer the main question in \Cref{sect:reducibility}.

\subsection{Cross-trees}

When considering \(\Pi^0_1\)-classes for the space \(\X\), it is natural to consider cross-trees which play a role analogous to binary trees in the case of the Cantor space.

\begin{definition}
    We extend the \emph{prefix relation} on strings \(\preccurlyeq\) to tuples of strings, in the natural way. More precisely, for any \(n\in\NN\), given integers \(k_0, \ldots, k_{n-1}\), and two tuples \(\sigma\coloneqq(\sigma_0, \ldots, \sigma_{n-1}), \tau\coloneqq(\tau_0, \ldots, \tau_{n-1})\in\prod_{i<n} k_i^{<\NN}\), we have 
    \(\sigma\preccurlyeq\tau\) if and only if \(\forall i<n, \sigma_i\preccurlyeq\tau_i\). For any \(k\in\NN\), the \emph{empty string} of \(k^{<\NN}\) is denoted by \(\varepsilon\), and we abuse the notation to also denote any tuple \((\varepsilon, \ldots, \varepsilon)\).

    Moreover, we define the set of its \emph{infinite extensions} by \([\sigma]\coloneqq[\sigma_0]\times\ldots\times[\sigma_{n-1}]\), and its \emph{length} by \(|\sigma|\coloneqq\max\{|\sigma_i| : i<n\}\).
\end{definition}

\begin{definition}
    A class \(\P\subseteq\X\) is \(\Pi^0_1\) if there is a \ce set \(W\subseteq\X_{<\NN}\) such that
    \[\overline{\P}=\bigcup_{\chi\in W} [\chi]\]
\end{definition}

\begin{definition}
    A \emph{cross-tree} is a set \(T\subseteq\X_{<\NN}\) which is downward-closed for the prefix relation \(\preccurlyeq\), and such that \(\forall (\rho, \sigma)\in T, \forall i\lt j\lt r, |\sigma_i|=|\sigma_j|\text{ and }|\sigma|\leqslant|\rho|\). The \emph{height} of \(T\) is \(h(T)\coloneqq\max\{|\chi|:\chi\in T\}\)

    The class of its (infinite) \emph{paths} is defined as 
    $$[T]\coloneqq\big\{(X, Y)\in\X: \forall n, (X\uh_n, Y\uh_n)\in T\big\}$$
    Moreover, given a string \(\rho\in\X_{<\NN}(0)\), we define the tree \(T[\rho]\coloneqq\{\sigma\in\X_{<\NN}(1): (\rho,\sigma)\in T\}\), which is finite since \(|\sigma|\leqslant|\rho|\). A cross-tree \(T\) is said to be \emph{right-pruned} if \(\forall\rho\in\X_{<\NN}(0), T[\rho]\) is pruned, \ie all the leaves of $T[\rho]$ have length $|\rho|$. Finally, for any \(N\in\NN\) we define \(T\uh_N\coloneqq\{\chi\in T : |\chi|\leqslant N\}\).
\end{definition}

\begin{lemma}
    A class \(\P\subseteq\X\) is \(\Pi^0_1\) if and only if there is a computable cross-tree \(T\subseteq\X_{<\NN}\) such that \([T]=\P\).
\end{lemma}
\begin{proof}
    Let \(T\subseteq\X_{<\NN}\) be a computable cross-tree. The set \([T]\) is \(\Pi^0_1\) as its complement is the set \(\bigcup_{\chi\in W} [\chi]\) where \(W\coloneqq \X_{<\NN}\smallsetminus T\) is computable.

    Now let \(\P\) be a \(\Pi^0_1\) class whose complement is \(\bigcup_{(\rho, \sigma)\in W} [\rho]\times[\sigma]\). Consider the cross-tree \(T\) such that 
    \((\rho, \sigma)\in T\iff \forall\mu\preccurlyeq\rho, \forall\tau\preccurlyeq\sigma, (\mu, \tau)\notin W[|\rho|]\text{ and }|\rho|\geqslant|\sigma|\). It is computable and its paths are exactly the elements of \(\P\).
\end{proof}

\begin{remark}
    In the rest of the document, every notion or proposition related to a class \(\P\subseteq\X\), also holds for a computable cross-tree \(T\subseteq\X_{<\NN}\) instead, by considering its associated class \([T]\).
\end{remark}

\subsection{Left-fullness}

The notion of left-fullness below is a sufficient notion of largeness such that any left-full \(\Pi^0_1\)-class contains multiple members satisfying some constraints. 
A key factor also lies in the fact that the complexity of this notion is only \(\Pi^0_1\).

\begin{definition}
    A class \(\P\subseteq\X\) is \emph{left-full} below \((\rho, \sigma)\in\X_{<\NN}\) if
    \[\forall X\in [\rho], \exists Y\in[\sigma], (X, Y)\in\P\]

    
    Moreover, for any integer \(N\in\NN\), we say that a finite tree $T \subseteq\X_{\leqslant N}$ is \emph{left-full} below $(\rho, \sigma) \in T$ if for every leaf $(\mu, \tau) \in T$, $|\mu| = N$, and for every~$\mu\in\X_N(0)$ extending $\rho$, there is some~$\tau \in\X_N(1)$ extending $\sigma$, such that $(\mu, \tau) \in T$.

    We simply say ``left-full" to signify ``left-full below \((\varepsilon, \varepsilon)\)".
\end{definition}

The above definition for finite trees is motivated by the following lemma. In particular it shows that \(T\) is left-full below \((\rho, \sigma)\) if and only if for any \(N\in\NN\), \(T\uh_N\) is left-full below \((\rho, \sigma)\).

\begin{lemma}\label{left-full-p01}
    Let \((\rho, \sigma)\in\X_{<\NN}\), and \(\P\subseteq\X\) be a \(\Pi^0_1\) class, whose associated computable cross-tree is \(T\).
    The statement
    \begin{equation*}
        \P\text{ is left-full below } (\rho, \sigma) \tag{a}\label{eq:a}
    \end{equation*}
    is equivalent to
    \begin{equation*}
        \forall\mu\succcurlyeq\rho, \exists \tau\succcurlyeq\sigma, |\tau|=|\mu|\text{ and }(\mu, \tau)\in T \tag{b}\label{eq:b}
    \end{equation*}
    Moreover, if \(T\) is right-pruned, then the statement is equivalent to
    \begin{equation*}
        \forall\mu\succcurlyeq\rho, (\mu, \sigma)\in T \tag{c}\label{eq:c}
    \end{equation*}
\end{lemma}
\begin{proof}
    We first show (\ref{eq:a})\(\implies\)(\ref{eq:b}). Let \(\mu\succcurlyeq\rho\). Consider \(X\in[\mu]\subseteq[\rho]\). By (\ref{eq:a}), there is \(Y\in[\sigma]\) such that \((X, Y)\in\P\). Thus, in particular, we have that \((\mu, Y\uh_{|\mu|})\in T\). Hence \(\tau\coloneqq Y\uh_{|\mu|}\) is the desired witness.
    
    For the converse, let \(X\in [\rho]\). We want to find \(Y\in [\sigma]\) such that \((X, Y)\in\P\). Consider the set \(S\coloneqq\{\tau\succcurlyeq\sigma:(X\uh_{|\tau|}, \tau)\in T\}\). Since \(T\) is a cross-tree, \(S\) is a finitely-branching cross-tree, with root \(\sigma\). Moreover, it is infinite because of (\ref{eq:b}). Thus, by König's lemma, there is \(Y\in 2^\NN\) such that \(\forall \ell, (X\uh_\ell ,Y\uh_\ell)\in T\). Hence \((X, Y)\in\P\).
    
    Last, we show (\ref{eq:c})\(\Leftrightarrow\)(\ref{eq:b}). For all \(\mu\succcurlyeq\rho\), if there is \(\tau\succcurlyeq\sigma\) such that \( |\tau|=|\mu|\text{ and }(\mu, \tau)\in T\), then in particular \((\mu, \sigma)\in T\) since \(T\) is a cross-tree.
    As for the converse, consider \(\mu\succcurlyeq\rho\). We have \((\mu, \sigma)\in T\), and since the cross-tree \(\{\tau: (\mu, \tau)\in T\}\) is pruned, it means there is \(\tau\succcurlyeq\sigma\) in it of size \(|\mu|\) and such that \((\mu, \tau)\in T\).
\end{proof}

The following lemma shows how left-fullness is preserved when extending or shortening the stems. Note that the first and second components do not play a symmetric role.

\begin{lemma}\label{left-full-ext}
    Let \(T\) be a computable cross-tree left-full below \((\rho, \sigma)\in\X_{<\NN}\). Then
    \begin{itemize}
        \item[(a)] \(\forall\widehat{\rho} \succcurlyeq\rho, \forall \widehat{\sigma} \preccurlyeq \sigma, 
        T\text{ is left-full below }(\widehat{\rho}, \widehat{\sigma})\)
        \item[(b)] For every~$n \in \NN$, there are some $\widehat{\rho} \succcurlyeq \rho$ and $\widehat{\sigma} \succcurlyeq \sigma$
        such that $|\widehat{\sigma}| = n$ and $T$ is left-full below $(\widehat{\rho}, \widehat{\sigma})$
    \end{itemize}
\end{lemma}
\begin{proof}
    Item~\((a)\) can be proven directly from the definition of left-full.
    For Item~\((b)\), we can suppose WLOG that $|\rho| \geq n$ and $|\sigma| \leq n$, thanks to the previous item. Consider all the extensions of \(\sigma\) of length \(n\), denoted \(\sigma_0, \ldots, \sigma_{k-1}\) for some \(k\in \NN\), and define \(\rho_0\coloneqq \rho\). If \(T\) is left-full below \((\rho_0, \sigma_0)\) then the assertion is proven. Otherwise, by \Cref{left-full-p01}, it means \(\exists\rho_1\succcurlyeq\rho_0, \forall\tau\succcurlyeq\sigma_0, |\tau|=|\rho_1|\implies (\rho_1, \tau)\notin T\). Now we consider the pair \((\rho_1, \sigma_1)\) to see if \(T\) is left-full below it. In case it is not, use \Cref{left-full-p01} as in the previous case. Proceed inductively like this for every \(\sigma_j\), where \(j<k\). If at some point \(T\) is left-full below \((\rho_j, \sigma_j)\) then the assertion is proven. Otherwise, it means we have built a sequence of string \(\rho_0\preccurlyeq\ldots\preccurlyeq\rho_k\) such that
    \(\forall j<k, \forall \tau\succcurlyeq\sigma_j, |\tau|=|\rho_{j+1}|\implies(\rho_{j+1}, \tau)\notin T\), and since \(T\) is a tree we even have \(\forall j<k, \forall \tau\succcurlyeq\sigma_j, |\tau|=|\rho_k|\implies(\rho_k, \tau)\notin T\). But the latter statement contradicts the fact that \(T\) is left-full below \((\rho, \sigma)\), by considering the string \(\rho_k\) in \Cref{left-full-p01}. 
\end{proof}

\subsection{Parameterized theorems}

Most of the computability-theoretic constructions of solutions to Ramsey-type theorems are done by variants of Mathias forcing, with reservoirs belonging to some Turing ideal containing only weak sets. The combinatorics of the statement usually require some closure properties on this ideal. For example, to construct solutions to computable instances of Ramsey's theorem for pairs, or to arbitrary instances of the pigeonhole principle, one requires the ideal to be a Scott ideal, that is, a model of~$\WKL$ (see~\cite{jockusch2009ramsey,seetapun1995strength}). One must then prove some basis theorem for~$\WKL$ to construct Scott ideals with only weak sets.

In our case, we shall need another closure property, yielding the notion of cross-constraint ideal. The main constructions of this section will be parameterized by cross-constraint ideals, whose existence will be proven in \Cref{sec:basis}.

\begin{definition}
    Let $X$ be an infinite set. A pair of instances \((f, g)\) of \(\RT^1_k\) is \emph{finitely compatible on $X$} if for all color \(i<k\) the set \(X \cap f^{-1}(i)\cap g^{-1}(i)\) is finite. Whenever $X = \NN$, we simply say that \((f, g)\) is finitely compatible. Also, note that the negation of ``finitely compatible" is ``infinitely compatible"
\end{definition}

\begin{definition}
    The \emph{cross-constraint problem}, denoted \(\CC\), is the problem whose instances are left-full cross-trees \(T\subseteq\X_{<\NN}\), and a solution to \(T\) is a pair of paths \((X^i, Y^i)_{i<2}\) such that
    \((X^0, X^1)\) is finitely compatible, and for all \(s<r\), \((Y^0_s, Y^1_s)\) is infinitely compatible.
\end{definition}

The following notion of cross-constraint ideal simply says that $\M$ is an $\omega$-model of~$\CC$, where an $\omega$-model is a structure of second-order arithmetic in which the first-order part consists of the standard integers. An $\omega$-model is then fully characterized by its second-order part, and is therefore often identified to it. The second-order part of an $\omega$-structure is a model of~$\RCA_0$ iff it is a Turing ideal.

\begin{definition}
     A \emph{cross-constraint ideal} is a Turing ideal \(\M\subseteq\kP(\NN)\) such that, any instance \(T \in \M\) of \(\CC\) has a solution \((X^i, Y^i)_{i<2}\) such that \((X^0, Y^0)\oplus (X^1, Y^1)\in\M\).
\end{definition}

Last, we define a notion of hyperimmunity for $k$-colorings~$f$, which is an intermediate notion between Cohen genericity and hyperimmunity. It implies in particular that for every~$j < k$, $\{ x : f(x) \neq j \}$ is hyperimmune, but in a dependent way.

\begin{definition}
    An instance \(f\) of \(\RT^1_k\) is \emph{hyperimmune} relative to \(D\subseteq\NN\) if for every \(D\)-computable array of \(k\)-tuples of mutually disjoint finite sets \(\{\vec{F_n}\}_{n\in\NN}\) such that \(\min\bigcup_{j<k} F_{n, j} > n\), there is \(N\in\NN\) such that \(\forall j<k, F_{N, j}\subseteq f^{-1}(j)\)
\end{definition}

In other words, an instance $f$ is hyperimmune relative to~$D$ if for every $D$-computable sequence $g_0, g_1, \dots$ of partial $k$-valued functions with finite support, such that $\forall n \in \NN, \min \dom g_n > n$, then $f$ is a completion of some~$g_n$.
We are now ready to prove our first parameterized theorem.

\begin{theorem}[{Liu~\cite[Theorem 2.1]{liu2023coding}}]\label[theorem]{param-1}
    Let \(\M\) be a countable cross-constraint ideal and let \(f\in\X(0)\) be hyperimmune relative to any element of \(\M\), then for any \(g\in\X(1)\) there is a solution \(\vec{G}\) of \(g\) which, for any \(Z\in\M\), does not \(Z\)-compute any solution of \(f\).
\end{theorem}
\begin{proof}
    The set \(\vec{G}\) is constructed by a variant of Mathias forcing, using conditions of the form \(\big((\vec{F_\alpha})_{\alpha\in 2^r}, \vec{A}\big)\) where 
    \begin{itemize}
        \item \(\vec{F_\alpha}\) is an \(r\)-tuple of finite sets \(g\)-homogeneous for the colors \(\alpha\), \ie \(\forall s<r, F_{\alpha, s}\subseteq g_s^{-1}(\alpha(s))\)
        \item \(\vec{A}\) is an \(r\)-tuple of infinite sets in \(\M\) such that \(\forall\alpha\in 2^r, \forall s<r, \min A_s>\max F_{\alpha, s}\)
    \end{itemize}
    The idea is that we do not know in advance what the colors of homogeneity will be for the solution being constructed, so we build all the possibilities in parallel, with \(\alpha\) indicating the colors, \ie for any \(i<r\), the set \(F_{\alpha, i}\) is \(g_i\)-homogeneous for the color \(\alpha(i)\).
    \bigskip
    
    A condition $\big((\vec{E_\alpha})_{\alpha\in 2^r}, \vec{B}\big)$ \emph{extends} another \(\big((\vec{F_\alpha})_{\alpha\in 2^r}, \vec{A}\big)\) if, for every \(\alpha\in 2^r\) and every \(s<r\), we have \(E_{\alpha, s}\supseteq F_{\alpha, s}\), \(B_s\subseteq A_s\), and \(E_{\alpha, s}-F_{\alpha, s}\subseteq A_s\). 
    \bigskip

    Every sufficiently generic filter~$\F$ for this notion of forcing induces a family of sets $(G^\F_{\alpha, s})_{\alpha \in 2^r, s < r}$ defined by
    $$
    G^\F_{\alpha, s} = \bigcup \{F_{\alpha,s} : \big((\vec{F_\alpha})_{\alpha\in 2^r}, \vec{A}\big) \in \F\}  
    $$
    Given an initial segment $F_{\alpha, s}$ of a condition $\big((\vec{F_\alpha})_{\alpha\in 2^r}, \vec{A}\big)$, it is not necessarily possible to find an extension $\big((\vec{E_\alpha})_{\alpha\in 2^r}, \vec{B}\big)$ with $|E_{\alpha,s}| > |F_{\alpha,s}|$,
    since it might be the case that $A_s\cap g^{-1}_s(\alpha(s))$ is empty. Thus, for any sufficiently generic filter~$\F$, the set $G^\F_{\alpha, s}$ might not be infinite. However, there must necessarily exist some $\alpha \in 2^r$ such that every $G^\F_{\alpha, s}$, $s < r$ is infinite.
    Given a condition \(c = \big((\vec{F_\alpha})_{\alpha\in 2^r}, \vec{A}\big)\) and a coloring \(h\in\X(1)\), define the set 
    $$V_h(c)\coloneqq\{\alpha\in 2^r : \forall s<r, A_s\cap h^{-1}_s(\alpha(s))\neq\emptyset\}$$ 
    of ``valid" combinations.
    Note that if $d\leqslant c$, then $V_h(d)\subseteq V_h(c)$. Moreover, $V_g(c) \neq \emptyset$ for every condition~$c$. A ``valid" combination of a condition \(c\) allows us to find an extension, as the following lemma shows.
    \begin{lemma}\label[lemma]{lem:param-1-forcing-inf}
        For all conditions $c\coloneqq \big((\vec{F_\alpha})_{\alpha\in 2^r}, \vec{A}\big)$, all~$\beta \in V_g(c)$ and all~$s < r$, there is an extension $\big((\vec{E_\alpha})_{\alpha\in 2^r}, \vec{B}\big)$ such that $|E_{\beta, s}| > |F_{\beta, s}|$.
    \end{lemma}
    \begin{proof}
        Let \(\beta\in V_g(c)\) and \(s<r\). By definition of \(\beta\) consider \(n\in A_s\cap g^{-1}_s(\alpha(s))\). Define \((\vec{E_\alpha})_{\alpha\in 2^r}\) to be equal to \((\vec{F_\alpha})_{\alpha\in 2^r}\), except for \(E_{\beta, s}\coloneqq F_{\beta, s}\cup\{n\}\). Accordingly, define \(\vec B\) to be equal to \(\vec A\), except for \(B_s\coloneqq A_s-\{0, 1, \dots, n\}\). Then $\big((\vec{E_\alpha})_{\alpha\in 2^r}, \vec{B}\big)$ is a condition which satisfies the lemma. 
    \end{proof}

    Fix an enumeration of Turing functionals \((\Psi_e)_{e\in\NN}\). For any \(e\in\NN\), \(j<3\) and \(Z\in\M\) let
    \[\R_{e, j}^Z\coloneqq \Psi_e^{Z\oplus\vec{G}}\text{ is not an infinite subset of }f^{-1}(j)\]
    
    \begin{lemma}\label[lemma]{lem:param-1-forcing-R}
        For any \(2^r\)-tuple of integers \((e_\alpha)_{\alpha\in 2^r}\), any \(2^r\)-tuple of colors $(j_\alpha)_{\alpha \in 2^r}$, any \(Z\in\M\), and any condition \(c\), there is an extension forcing \(\bigvee_{\alpha\in V_g(c)} \R_{e_\alpha,j_{\alpha}}^Z\).
    \end{lemma}
    \begin{proof}
        Let \(c\coloneqq \big((\vec{F_\alpha})_{\alpha\in 2^r}, \vec{A}\big)\) be a condition, and define \(U(c)\) to be the \(\Pi^0_1(\vec A)\)-class whose elements are colorings \(\widetilde g\) such that 
        $$\forall s<r, \forall j<3, A_s\cap g_s^{-1}(j)=\emptyset\implies A_s\cap \widetilde{g_s}^{-1}(j)=\emptyset$$ in other words \(U(c)\coloneqq\{\widetilde g : V_{\widetilde g}(c)=V_g(c)\}\). Note that it is non-empty since \(g\in U(c)\).
        This will ensure that the \(\widetilde g\) we consider later have the same behavior as \(g\) regarding \(\vec A\).
        
        Moreover, given \(\alpha\in 2^r\), an \(r\)-tuple of finite sets \(\vec G\) \emph{satisfies} \((\vec{F_\alpha}, \vec A)\) if \(\forall s<r, G_s\supseteq F_{\alpha, s}\text{ and }G_s-F_{\alpha, s}\subseteq A_s\)

        For every \(n\in\NN\) consider the class \(\Q_n\) whose elements are colorings \(\widetilde{f}\in\X(0)\) such that
    \begin{align*}
        &\exists \widetilde{g} \in U(c),\\
        &\forall\alpha\in V_g(c), \forall\vec{G}\text{ satisfying }(\vec{F_\alpha}, \vec{A})\\
        &\text{if }\vec{G}\text{ is }\widetilde{g}\text{-homogeneous for the colors }\alpha,\\
        &\text{then }\Psi_{e_\alpha}^{Z\oplus\vec{G}}\cap\rrbracket n, +\infty\llbracket \subseteq\widetilde{f}^{-1}(j_{\alpha})
    \end{align*}
    Note that \(\Q_n\) is a \(\Pi^0_1(\vec{A})\) class uniformly in \(n\). Indeed, the above formula is \(\Pi^0_1(\vec{A})\), because by compactness, the existence of \(\widetilde g\) is equivalent to finding an approximation in \(\X_{m}(1)\) for any length \(m\in\NN\). Moreover, the set $V_g(c)$ depends on~$g$ which might be of arbitrary complexity, but since it is finite, it does not affect the complexity of the formula.
    \bigskip
    
    \textbf{Case 1}, \(\exists n, \Q_n=3^\NN\). Thus fix such an \(n\) and consider the class
    \(\P\coloneqq\{(\widetilde{f}, \widetilde{g})\in\X: \widetilde{g}\text{ is a witness of }\widetilde{f}\in\Q_n\}\). Since $\P$ is a left-full \(\Pi^0_1(\vec{A})\)-class and \(\M\) is a cross-constraint ideal, there are paths \((X^i, Y^i)_{i<2}\in\P^2\) such that
    \begin{itemize}
        \item \((X^0, X^1)\) is finitely compatible
        \item for any \(s<r\), \((Y^0_s, Y^1_s)\) is infinitely compatible on $A_s$
        \item \((X^0, Y^0)\oplus(X^1, Y^1)\in\M\)
    \end{itemize}
    From the second item, define \(\beta\in 2^r\) to be colors such that \(\forall s<r, (Y^0_s)^{-1}(\beta(s))\cap (Y^1_s)^{-1}(\beta(s))\cap A_s\text{ is infinite}\). Note that $\beta \in V_{Y_0}(c) \cap V_{Y_1}(c)$. Since $g, Y^0, Y^1 \in U(c)$, then $V_{Y_0}(c) = V_{Y_1}(c) = V_g(c)$, so $\beta \in V_g(c)$.
    For each \(s<r\), let \(B_s = (Y^0_s)^{-1}(\beta(s))\cap (Y^1_s)^{-1}(\beta(s))\cap A_s\), and define \(\vec{B}\coloneqq (B_0, \ldots, B_{r-1})\). We claim that \(((\vec{F_\alpha})_{\alpha\in 2^r}, \vec{B})\) is the extension we are looking for. 
    Indeed, since \((X^i, Y^i)_{i<2}\in\P^2\), then for every \(\vec{G}\) compatible with \((F_{\beta}, \vec{B})\), we have \(\Psi^{Z\oplus\vec{G}}_{e_\beta}\cap\rrbracket n, +\infty\llbracket\,\subseteq {X_0}^{-1}(j_{\beta})\cap {X_1}^{-1}(j_{\beta})\). Since \({X_0}^{-1}(j_{\beta})\cap {X_1}^{-1}(j_{\beta})\) is finite, the requirement \(\R_{e_\beta,j_\beta}^Z\) is satisfied, so \(\bigvee_{\alpha\in V_g(c)} \R_{e_\alpha,j_\alpha}^Z\) is satisfied as well.
    \bigskip
    
    \textbf{Case 2}, \(\forall n, \Q_n\neq 3^\NN\). 
    
    In which case, for each \(n\in\NN\), there is some \(f_n\notin\Q_n\) and \(\ell_n\in\NN\) such that
    \begin{align*}
        &\forall\widetilde{g}\in U(c),\\
        &\exists\beta\in V_g(c), \exists\vec{H}\text{ satisfying }(\vec{F_\beta}, \vec{A}),\\
        &\vec{H}\text{ is }\widetilde{g}\text{-homogeneous for the colors }\beta\\
        &\text{ and }\Psi_{e_\beta}^{Z\oplus\vec{H}}\cap\rrbracket n, \ell_n\llbracket\nsubseteq f_n^{-1}(j_{\beta})
    \end{align*}
    
    This implies that there is an \(\vec{A}\)-computable array \((E_{n, 0}, E_{n, 1}, E_{n, 2})_{n\in\NN}\) such that
    \begin{align*}
        &\forall\widetilde{g}\in U(c),\\
        &\exists\beta\in V_g(c), \exists\vec{H}\text{ satisfying }(\vec{F_\beta}, \vec{A}),\\
        &\vec{H}\text{ is }\widetilde{g}\text{-homogeneous for the colors }\beta\\
        &\text{ and }\Psi_{e_\beta}^{\vec{A}\oplus\vec{H}}\cap\rrbracket n, \ell_n\llbracket\nsubseteq E_n(j_{\beta})
    \end{align*}
    Indeed, for each \(n\), the set \(\rrbracket n, \ell_n\llbracket\) can be partitioned into \((\{m:f_n(m)=j\})_{j<3}\), so it is sufficient to search for a coloring satisfying the above properties, as the search must end.

    The coloring \(f\) is hyperimmune relative to \(\vec{A}\), because \(\vec{A}\in\M\). Hence there is some \(n\in\NN\) such that \(\forall j<3, E_{n, j}\subseteq f^{-1}(j)\).
    Moreover, we have \(E_{n, j}=f^{-1}(j)\cap\rrbracket n, \ell_n\llbracket\). So, by considering \(g\), we obtain
    \begin{align*}
        &\exists\beta\in V_g(c), \exists\vec{H}\text{ satisfying }(\vec{F_\beta}, \vec{A}),\\
        &\vec{H}\text{ is }g\text{-homogeneous for the colors }\beta\\
        &\text{ and }\Psi_{e_\beta}^{Z\oplus\vec{H}}\cap\rrbracket n, \ell_n\llbracket\nsubseteq f^{-1}(j_{\beta})
    \end{align*}

    Finally, the extension we are looking for is \(\big(\{\vec{F_\alpha}\}_{\alpha\neq\beta\in 2^r}\cup\{\vec{H}\}, \vec{A}-\{0, \ldots, \max\vec{H}\}\big)\). This completes the proof of \Cref{lem:param-1-forcing-R}.
    \end{proof}

    Let $c_0$ be a condition such that $V_g(c_0)$ is minimal for inclusion. Let~$V = V_g(c_0)$.
    Let $\F$ be a sufficiently generic filter containing~$c_0$. In particular, $V = V_g(c)$ for every $c \in F$.
    By \Cref{lem:param-1-forcing-inf}, for every~$\alpha \in V$ and every~$s < r$, $G^\F_{\alpha, s}$ is infinite.
    Moreover, by \Cref{lem:param-1-forcing-R}, for every~$Z \in \M$, for every $2^r$-tuple of integers $(e_{\alpha})_{\alpha \in 2^r}$ and every $2^r$-tuple of colors $(j_{\alpha})_{\alpha \in 2^r}$, $\vec{G}^\F$ satisfies $\bigvee_{\alpha\in V} \R^Z_{e_\alpha,j_\alpha}$.
    By a pairing argument, for every~$Z \in \M$, there is some~$\alpha \in V$ such that $Z \oplus \vec{G}^\F_\alpha$ does not compute any infinite $f$-homogeneous set. Since~$\M = \{Z_0, Z_1, \dots \}$ is countable, by the infinite pigeonhole principle, there exists some~$\alpha \in V$ such that for infinitely many~$n \in \NN$, $Z_0 \oplus \dots \oplus Z_n \oplus \vec{G}^\F_\alpha$ does not compute any infinite $f$-homogeneous set. By downward-closure of this property under the Turing reduction, it holds for every~$n$. This completes the proof of \Cref{param-1}.
\end{proof}

Our first parameterized theorem has applications in terms of strong non-reducibility between non-computable instances of~$\RT^1_{k+1}$ and $(\RT^1_k)^{*}$ (\Cref{thm:rt13-soc-rt12}) and non-reducibility between computable instances of~$\SRT^2_{k+1}$ and $(\SRT^2_k)^{*}$ (\Cref{separation2}). We now prove a second parameterized theorem which enables us to prove separations between computable instances of Ramsey's theorem for pairs. Note that in the following theorem, the colorings $g_0, \dots, g_{r-1}$ are required to belong to~$\M$, contrary to \Cref{param-1}.

\begin{theorem}\label{param-2}
    Let \(\M\) be a countable cross-constraint ideal such that $\M \models \COH$ and let \(f : \NN \to 3\) be hyperimmune relative to any element of \(\M\), then for any $r \in \NN$ and any \(g_0, \dots, g_{r-1} : [\NN]^2 \to 2\) in~$\M$, there are infinite $g_i$-homogeneous sets~$G_i$ for every~$i < r$, such that $\bigoplus_{i<r} G_i$ does not compute any infinite $f$-homogeneous set.
\end{theorem}
\begin{proof}
    First, consider the sequence of sets \(\vec{R}\coloneqq(R_{x, j, i})_{x\in\NN, j<2, i<r}\) defined by \(R_{x, j, i}\coloneqq\{y:g_i(x, y)=j\}\). This sequence is in \(\M\), because \(\forall i<r, g_i\in\M\). And since \(\M\models\COH\), there is an infinite \(\vec{R}\)-cohesive set \(C\coloneqq\{c_0 < c_1 < \ldots\}\in\M\). By choice of~$\vec{R}$, for each \(i<r\), the coloring \(g_i\uh_{[C]^2}\) is stable
    Indeed, for \(x\in\NN\) and \(i<r\), there is exactly one \(j<2\) for which \(C\subseteq^* R_{x, j, i}\), and so this implies that \(\lim_{y\in C} g_i(x, y)=j\).
    
    Now let \(h_i:\NN\to 2, n\mapsto\lim_m g_i(c_n, c_m)\), and \(\vec{h}\coloneqq(h_0, \ldots, h_{r-1})\).
    By \Cref{param-1} with \((f, (h_i)_{i<r})\), there are infinite \(\vec{h}\)-homogeneous sets \(\vec{H}\) such that, for any \(Z\in\M\), \(\vec{H}\oplus Z\) does not compute an infinite \(f\)-homogeneous set.
    In particular this is true for \(Z\coloneqq \bigoplus_{i<r} C_i\), and since for any \(i<r\), \(H_i\oplus C_i\) computes an infinite \(g_i\)-homogeneous set \(G_i\), we deduce that \(\bigoplus_{i<r} G_i\) does not compute any infinite \(f\)-homogeneous set.
\end{proof}

\section{Cross-constraint basis theorems}\label[section]{sec:basis}

As mentioned before, the two main theorems of \Cref{sect:general-framework} are parameterized by cross-constraint ideals, which are themselves built using iterated applications of the cross-constraint principle ($\CC$). 
In this section, we prove various basis theorems for $\CC$, namely, the $\Delta^0_2$, low, cone avoidance, and non-$\Sigma^0_1$ preservation basis theorems. The $\Delta^0_2$ and cone avoidance basis theorems for~$\CC$ were previously proven by Liu~\cite{liu2023coding}, but we give a new proof of the cone avoidance basis theorem which resembles more its classical counterpart for $\Pi^0_1$ classes.


\subsection{Conditions}

The most famous basis theorems for $\Pi^0_1$ classes are all proven using effective versions of forcing with binary trees. Similarly, all the basis theorems for $\CC$ in this section will be proven with effective variants of the same notion of forcing that we now describe.

\begin{definition}
    For \(k\in\NN\) and \(m\leqslant n\in\NN\), given chains \(\rho^0, \rho^1\in k^m\), \(\mu^0, \mu^1\in k^n\), and \(\tau^0, \tau^1\in k^{n-m}\) such that \(\forall i<2, \mu^i=\rho^i\cdot\tau^i\).
    We say that \((\mu^0, \mu^1)\) is \emph{completely incompatible over \((\rho^0, \rho^1)\)} if \(\forall i<n-m, \tau^0(i)\neq\tau^1(i)\).
    When \((\rho^0, \rho^1)=(\varepsilon, \varepsilon)\), we simply say \emph{completely incompatible}.
\end{definition}


\begin{definition}
    A \emph{condition-tuple} for a class \(\P\subseteq\X\) is a tuple $(\rho^i, \sigma^i)_{i < 2}\in\X_{<\NN}^2$ such that \(\P\) is left-full below \((\rho^i, \sigma^i)\) for both \(i<2\), and \(|\rho^0|=|\rho^1|\).
    A condition-tuple $(\widehat{\rho}^i, \widehat{\sigma}^i)_{i < 2}$ \emph{extends} another $(\rho^i, \sigma^i)_{i < 2}$, written $(\widehat{\rho}^i, \widehat{\sigma}^i)_{i < 2} \leq (\rho^i, \sigma^i)_{i < 2}$, if
    \begin{enumerate}
        \item for both \(i<2\), $(\widehat{\rho}^i, \widehat{\sigma}^i) \succcurlyeq (\rho^i, \sigma^i)$
        \item \((\widehat\rho^0, \widehat\rho^1)\) is completely incompatible over \((\rho^0, \rho^1)\)
    \end{enumerate}
\end{definition}

\begin{remark}
     A pair \((\rho^i, \sigma^i)\) of a condition-tuple is seen as a finite approximation of some \((X^i, Y^i)\in\P\) we wish to build, \ie of an element in \(\P\cap\left[(\rho^i, \sigma^i)\right]\), such that \((X^0, X^1)\) is completely incompatible over \((\rho^0, \rho^1)\).
\end{remark}

\subsection{\texorpdfstring{Cross-constraint $\Delta^0_2$ basis theorem}{Cross-constraint Delta02 basis theorem}}

The following $\Delta^0_2$ basis theorem is an effective analysis of the simplest known combinatorial proof of the cross-constraint problem. It was proven by Liu~\cite{liu2023coding}. We provide the original proof as a warmup for the next basis theorems and for the sake of completeness. 


\begin{theorem}[{Liu~\cite[Lemma 2.2]{liu2023coding}}]\label{delta-basis}
    Any left-full computable instance \(T\) of \(\CC\) has a solution \((X^i, Y^i)_{i<2}\) such that \((X^0, Y^0)\oplus(X^1, Y^1)\leqslant_T\emptyset'\).
\end{theorem}
\begin{proof}
    Let \(\P\coloneqq [T]\). 
    We build a \(\emptyset'\)-computable sequence of condition-tuples \((\rho^i_0, \sigma^i_0)_{i<2} \geq (\rho^i_1, \sigma^i_1)_{i<2} \geq \ldots\), such that, for \(i<2\) the functions \(X^i:=\bigcup_t \rho^i_t\) and \(Y^i:=\bigcup_t \sigma^i_t\) are the desired witnesses. To simplify notation, we simply say condition for condition-tuple.
    
    Given a condition \((\rho^i_t, \sigma^i_t)_{i<2}\), we want to extend it to \((\rho^i_{t+1}, \sigma^i_{t+1})_{i<2}\) such that for all \(s<r\), \((\sigma^0_{t+1, s}, \sigma^1_{t+1, s})\) is not completely incompatible over \((\sigma^0_{t, s}, \sigma^1_{t, s})\). 
    We proceed as follows:
    say \((\rho^i_t, \sigma^i_t)_{i<2}\) has been constructed, search for an extension \((\rho^i_{t+1}, \sigma^i_{t+1})_{i<2}\) such that the property above holds. By \Cref{left-full-p01}, this search requires the use of \(\emptyset'\) to check whether \(\P\) is left-full below \((\rho^i_{t+1}, \sigma^i_{t+1})\) or not (for each i).

    If there is always such an extension, then the construction is complete. Otherwise, for some condition \((\rho^i_t, \sigma^i_t)_{i<2}\), there is some \(s<r\) such that all extensions are completely incompatible over \((\sigma^0_{t, s}, \sigma^1_{t, s})\). Such a condition is said to \emph{exclude component \(s<r\)}. Nevertheless, by considering the following lemma, the construction can be completed.
    
    \begin{lemma}\label[lemma]{lem:delta-basis-excludes}
        If a condition \((\rho^i, \sigma^i)_{i<2}\) excludes component \(s<r\), then for every~$i < 2$, there is a coloring $Y \in 2^\NN$ such that for every $(\widehat\rho, \widehat\sigma) \leqslant (\rho^i, \sigma^i)$, if \(\P\) is left-full below $(\widehat\rho, \widehat\sigma)$, then $\widehat\sigma_s$ is compatible with $Y$.
    \end{lemma}
    \begin{proof}
        Fix \((\rho^i, \sigma^i)_{i<2}\) such that it excludes component \(s<r\) and fix~$i < 2$.
        
        Suppose first that for cofinitely many~$n \in \NN$, there exists a $k_n < 2$ such that for every $(\rho', \sigma') \leq (\rho^i, \sigma^i)$ for which \(\P\) is left-full, if $|\sigma'_s| > n$ then $\sigma'_s(n) = k_n$. Then let~$Y(n) \coloneqq k_n$. This is well-defined since by \Cref{left-full-ext}, there exist suitable extensions of every length.
        
        Suppose now that for infinitely many~$n \in \NN$, for every~$k < 2$, there is a $(\rho', \sigma') \leq (\rho^i, \sigma^i)$ for which \(\P\) is left-full and such that $\sigma'_s(n) \neq k$. Fix~$n > |\sigma^i_s|$, and for each~$k < 2$, let~$(\rho_k, \sigma_k) \leq (\rho^i, \sigma^i)$ be a condition extension such that $\sigma_{k,s}(n) \neq k$. 
        
        Since we are working in a 3-valued realm, there exists a $\widetilde{\rho} \succcurlyeq \rho^{1-i}$ of length greater than $\max(|\rho_0|, |\rho_1|)$ such that both $(\widetilde{\rho}, \rho_0)$ and $(\widetilde{\rho}, \rho_1)$ are completely incompatible over $(\rho^{1-i}, \rho^i)$. 
        Pick any pair $(\widehat{\rho}^{1-i}, \widehat{\sigma}^{1-i}) \leq (\widetilde{\rho}, \sigma^{1-i})$ for which \(\P\) is left-full and such that $|\widehat{\sigma}_s^{1-i}| > n$, this is possible by \Cref{left-full-ext}. Let~$k = \widehat{\sigma}_s^{1-i}(n)$. Then $(\widehat{\rho}^{1-i}, \widehat{\sigma}^{1-i}, \rho_k, \sigma_k)$ is an extension contradicting the fact that \((\rho^i, \sigma^i)_{i<2}\) excludes component \(s<r\)
    \end{proof}

    If at some point in the construction the condition \((\rho^i_t, \sigma^i_t)_{i<2}\) excludes the components of \(I\subseteq r\), then we restart the construction from the begining with \((\rho^i_0, \sigma^i_0)_{i<2}:=(\rho^0_t, \sigma^0_t, \rho^0_t, \sigma^0_t)\). The lemma ensures that \((Y^0_s, Y^1_s)\) will not be finitely compatible for all \(s\in I\), no matter the conditions selected for the sequence. And since \(r\) is a standard integer, the construction can only be restarted a finite number of times.
\end{proof}

\begin{remark}
The fact that $X^i$ is an instance of $\RT^1_3$ whereas each~$Y^i_s$ is an instance of~$\RT^1_2$ is exploited in the proof of \Cref{lem:delta-basis-excludes}, to ensure the existence of a 3-valued string $\widetilde{\rho}$ completely incompatible with two other strings simultaneously.
\end{remark}

\begin{corollary}\label{exists-arith-ideal}
    The class of all arithmetic sets is a cross-constraint ideal.
\end{corollary}

\subsection{Combinatorial lemmas}\label[section]{sect:combi-lemmas}

All the remaining basis theorems will involve some kind of first-jump control.
They will require a much more involved combinatorial machinery that we now develop. These combinatorics are all due to Liu~\cite{liu2023coding}, with a slightly different organization and terminology.

\begin{definition}
    For \(k\in\NN\) and \(m\leqslant n\in\NN\), given chains \(\rho^0, \rho^1\in k^{\leqslant m}\), a total function \(\varphi : k^m \to k^n\) \emph{preserves incompatibility over \((\rho^0, \rho^1)\)} if
    \begin{enumerate}
        \item \(\forall\mu\in k^m, \varphi(\mu) \succcurlyeq \mu\)
        \item For all \(\mu^0, \mu^1\in k^m\) such that \((\mu^0, \mu^1)\succcurlyeq(\rho^0, \rho^1)\), the pair \((\varphi(\mu^0), \varphi(\mu^1))\) is completely incompatible over \((\mu^0, \mu^1)\)
    \end{enumerate}
\end{definition}
\begin{remark}
    Note that for \(\widehat\rho^0, \widehat\rho^1\in k^{\leqslant n}\) extending \(\rho^0, \rho^1\) respectively, if \(\varphi\) preserves incompatibility over \((\rho^0, \rho^1)\) then it also preserves incompatibility over \((\widehat\rho^0, \widehat\rho^1)\).
\end{remark}

The following two lemmas are purely technical ones, used only locally to obtain the main combinatorial lemmas of this section.

\begin{lemma}[{Liu~\cite[Lemma 4.3]{liu2023coding}}]\label[lemma]{ext1}
    For any \(n\leqslant n'\leqslant m\in\NN\), \(\widehat\rho\in 3^m\) extending \(\rho\in 3^n\), and any map \(\psi:3^n\to 3^{n'}\) preserving incompatibility over \((\varepsilon, \varepsilon)\) such that \(\psi(\rho)\preccurlyeq\widehat\rho\), there is a map \(\varphi:3^n\to 3^m\) extending \(\psi\) and preserving incompatibility over \((\varepsilon, \varepsilon)\), such that \(\varphi(\rho)=\widehat\rho\).
\end{lemma}
\begin{proof}
    To give an intuition, if we just wanted to show the existence of a map \(\varphi:3^n\to 3^m\) preserving incompatibility over \((\varepsilon, \varepsilon)\), we could have taken the function which maps \(\eta\in 3^n\) to \(\eta\cdot a^{m-n}\) where \(a\coloneqq \eta(0)\).
    
    Fix~$\rho$, $\widehat\rho$ and \(\psi\) as in the statement of the lemma, and let~$\tau$ be such that $\rho \cdot \tau = \widehat\rho$.
    For every~$a < 3$, let~$\tau_a \in 3^{m-n'}$ be defined by $\tau_a(x) = \tau(x)+a-\rho(0) \mod 3$. In particular, $\tau_{\rho(0)} = \tau$ and for every~$a \neq b$, $\tau_a$ and $\tau_b$ are preserving incompatibility over \((\varepsilon, \varepsilon)\).
    
    Let~$\varphi$ be the function \(\varphi'\circ \psi\) where \(\varphi'\) maps \(\eta\in 3^n\) to \(\eta\cdot \tau_{\eta(0)}\). Note that $\varphi'(\rho) = \rho \cdot \tau_{\rho(0)} = \widehat\rho$. Moreover, if $\mu, \nu \in 3^n$ are completely incompatible, then $\psi(\mu)(0) \neq \psi(\nu)(0)$, hence $\tau_{\mu(0)}$ and $\tau_{\nu(0)}$ are completely incompatible. It follows that $\varphi(\mu) = \mu \cdot \tau_{\mu(0)}$ and $\varphi(\nu) = \nu \cdot \tau_{\nu(0)}$ are completely incompatible themselves.
\end{proof}

\begin{lemma}[{Liu~\cite[Lemma 4.4]{liu2023coding}}]\label{ext2}
    Let \(N\in\NN\) and \(f:3^{<\NN}\to \mathfrak{P}\big(\X_N(1)\big)\) be a total and non-increasing function. For every \(n\in\NN\), there is \(m\in\NN\) and a map \(\varphi:3^n\to 3^m\) preserving incompatibility over \((\varepsilon, \varepsilon)\), such that \(\forall \rho\in 3^n, \forall \nu\in\X_N(1)\) either \(\forall \widehat\rho\succcurlyeq\varphi(\rho), \nu\in f(\widehat\rho)\) or \(\nu\notin f(\varphi(\rho))\).

\end{lemma}

\begin{proof}
    We construct a finite sequence of integers \((n_s)_{s\leqslant p}\), for some \(p\in\NN\), along with a finite sequence of maps \((\varphi_s:3^{n_0}\to 3^{n_s})_{s\leq p}\) preserving incompatibility over \((\varepsilon, \varepsilon)\).

    For each \(\nu\in\X_N(1)\) and \(\rho\in 3^n\) there is a step to ensure that the map \(\varphi\) we construct satisfies the requirement \[\R_{\nu, \rho}\coloneqq \forall \widehat\rho\succcurlyeq\varphi(\rho), \nu\in f(\widehat\rho)\text{ or }\nu\notin f(\varphi(\rho))\]
    
    Let \(n_0\coloneqq n\) and $\varphi_0 : 3^{n_0} \to 3^{n_0}$ be the identity function. At step \(s<p\), consider \(\nu\in\X_N(1)\) and \(\rho\in 3^n\).
    If \(\forall\widehat\rho\succcurlyeq\varphi_s(\rho), \nu\in f(\widehat\rho)\) we are done by defining \(n_{s+1}\coloneqq n_s\) and \(\varphi_{s+1}\coloneqq\varphi_s\).
    Otherwise there is \(\eta\) extending \(\varphi_s(\rho)\) such that \(\nu\notin f(\eta)\). Define \(n_{s+1}\coloneqq|\eta|\), and use \Cref{ext1} to find a map \(\varphi_{s+1}:3^n\to 3^{n_{s+1}}\) extending \(\varphi_s\), preserving incompatibility over \((\varepsilon, \varepsilon)\), and such that \(\varphi_{s+1}(\rho)=\eta\). So, under the assumption made on $f$ we have that \(\forall\widehat\eta\succcurlyeq\eta, f(\widehat\eta)\subseteq f(\eta)\), and thus $\forall \widehat\eta\succcurlyeq\eta, \nu\notin f(\widehat\eta)$. Then as $\varphi_{s+1}$ extends $\varphi_s$, $\R_{\nu, \rho}$ is fulfilled.
    
    Finally, define \(m\coloneqq n_p\) and \(\varphi\coloneqq\varphi_p\). By the hypothesis made on the function \(f\), for every~$\nu\in\X_N(1)$, if $\nu\in f(\varphi(\rho))$ for some~$\rho \in 3^n$, then \(\forall \widehat\rho\succcurlyeq \varphi(\rho), \nu\in f(\widehat\rho)\)
\end{proof}

The following definition should be understood in the light of the first-jump control for $\Pi^0_1$ classes. When trying to construct an infinite path through an infinite binary tree~$T \subseteq 2^{<\omega}$, one must ensure that at every step, the node is \emph{extensible}, that is, the branch below the node is infinite. Being extensible is a $\Pi^0_1$ property, and therefore to obtain a good first-jump control, one must resort to an overapproximation: Given a set~$A \subseteq 2^{<\omega}$, instead of asking whether there is an extensible node in~$T \cap A$, one will ask whether there is a level in the tree such that every node at this level belongs to~$A$. Among the nodes at that level, at least one must be extendible. If~$A$ is $\Sigma^0_1$, then the former question is $\Sigma^0_2$, while the latter is $\Sigma^0_1$.

One can use a different technique, and ask whether the $\Pi^0_1$ class $\P$ of infinite subtrees of $T$ disjoint from~$A$ is empty. In particular, by considering $S \subseteq T$, the pruned subtree of~$T$ containing only extendible nodes, since~$S \not \in \P$, there is an extendible node in~$S \cap A$. Here again, this overapproximation is $\Sigma^0_1$. 

In the case of cross-constraint problems, a node $(\rho, \sigma)$ is extensible in a cross-tree~$T$ if $T$ is left-full below $(\rho, \sigma)$. The notion of~$T$-sufficiency below is therefore the counterpart of the $\Sigma^0_1$ question above, mutatis mutandis.

\begin{definition}[{Liu~\cite[Definition 4.22]{liu2023coding}}]
    Given a tree \(T\subseteq\X\), and a tuple \((\rho^i, \sigma^i)_{i<2}\in\X_{<\NN}^2\), a set \(A\subseteq\X_{<\NN}^2\) is \emph{$T$-sufficient over \((\rho^i, \sigma^i)_{i<2}\)} if, for every infinite subtree \(S\subseteq T\) for which \((\rho^i, \sigma^i)_{i<2}\) is a condition-tuple, there is a tuple \((\widehat\rho^i, \widehat\sigma^i)_{i<2}\in A\cap S^2\) such that \(\forall i<2, (\widehat\rho^i, \widehat\sigma^i)\succcurlyeq(\rho^i, \sigma^i)\), and \((\widehat\rho^0, \widehat\rho^1)\) is completely incompatible over \((\rho^0, \rho^1)\).
\end{definition}

Note that if~$A$ is $\Sigma^0_1$, then the statement \qt{$A$ is $T$-sufficient over \((\rho^i, \sigma^i)_{i<2}\)} is $\Sigma^0_1(T)$. The combinatorics for the cross-constraint problem are more complicated than the ones for weak K\"onig's lemma, and one cannot simply consider the pruned tree containing only extensible nodes. However, the following lemmas show that one can consider a weakly pruned tree in which every node is almost extensible, in the sense that every node can be extended into a node below which the cross-tree is left-full.

The following lemma uses compactness to give a finite cross-tree version of $T$-sufficiency.
Recall that the notion of left-fullness was extended to finite trees, which induces a notion of condition-tuple.
Given a set $S \subseteq \X_{<\NN}$, we write $\ell_N(S)$ for the set of pairs $(\rho, \sigma) \in S$ such that $|\rho| = |\sigma| = N$.

\begin{lemma}\label{konig-sufficient}
    Let \(T\subseteq\X\) be a cross-tree. If \(A\subseteq\X_{<\NN}^2\) is \(T\)-sufficient over \((\rho^i, \sigma^i)_{i<2}\in\X_{<\NN}^2\), then there is \(N\in\NN\) such that for every finite cross-subtree $S \subseteq T \cap \X_{\leq N}$ for which \((\rho^i, \sigma^i)_{i<2}\) is a condition-tuple, there is a tuple \((\widehat\rho^i, \widehat\sigma^i)_{i<2}\in A\cap S^2\) such that \(\forall i<2, (\widehat\rho^i, \widehat\sigma^i)\succcurlyeq(\rho^i, \sigma^i)\), and \((\widehat\rho^0, \widehat\rho^1)\) is completely incompatible over \((\rho^0, \rho^1)\).
\end{lemma}
\begin{proof}
    Consider the class \(\T\) of finite subtrees \(S\subseteq T\) whose leaves are all of the same length, such that \(S\) is left-full below \((\rho^i, \sigma^i)\) for both \(i<2\), and such that for all tuple \((\widehat\rho^i, \widehat\sigma^i)_{i<2}\in S^2\) which extends \((\rho^i, \sigma^i)\) and such that \((\widehat\rho^0, \widehat\rho^1)\) is completely incompatible over \((\rho^0, \rho^1)\), then \((\widehat\rho^0, \widehat\rho^1)\notin A\). 
    It forms a tree for the relation where \(S_0\leqslant S_1\) if and only if \(S_1\uh_{h(S_0)}=S_0\). 
    
    There is no infinite path in \(\T\), otherwise it would contradict the assumption that \(A\) is \(T\)-sufficient. Hence \(\T\) is finite thanks to König's lemma. In other words, there is \(N\in\NN\) such that, for any finite subtree \(S\subseteq T \cap \X_{\leq N}\) left-full below both \((\rho^i, \sigma^i)\), there is \((\widehat\rho^i, \widehat\sigma^i)\in A\cap S^2\) witnessing the \(T\)-sufficiency of~\(A\).
\end{proof}

The following lemma is the desired combinatorial lemma. 

\begin{lemma}[{Liu~\cite[Subclaim 4.7]{liu2023coding}}]\label{sufficient-imply-extendible}
    Let \(T\subseteq\X\) be a cross-tree. If \(A\subseteq\X_{<\NN}^2\) is \(T\)-sufficient over \((\rho^i, \sigma^i)_{i<2}\in\X_{<\NN}^2\) and closed under the suffix relation, then there is a condition-tuple in \(A\) which extends \((\rho^i, \sigma^i)_{i<2}\).

\end{lemma}
\begin{proof}
Let $N \in \NN$ witness \Cref{konig-sufficient}.
For \(i<2\), define the maps
    \begin{align*}
        f_i:3^{<\NN}&\to \mathfrak{P}\big(\X_{N-|\sigma^i|}(1)\big)\\
        \tau&\mapsto \left\{\nu: (\rho^i\cdot\tau, \sigma^i\cdot\nu)\in T\right\}
    \end{align*}
    And consider the map \(f:\tau\mapsto f_0(\tau)\cup f_1(\tau)\).
    
    Since \(T\) is left-full below \((\rho^i, \sigma^i)\) for both \(i<2\), then by \Cref{left-full-ext}, \(\forall\tau, f(\tau)\neq\emptyset\). Moreover, each \(f_i\) is non-decreasing since \(T\) is a cross-tree, so \(f\) is also non-decreasing.
    Thus we can apply \Cref{ext2} on \(f\) to obtain \(L\in\NN\) and a map \(\psi:3^{N-|\rho^0|}\to 3^L\) preserving incompatibility over \((\rho^0, \rho^1)\) such that, for any \(\tau\in \X_{N-|\rho^0|}(0)\) and any \(\nu\in\X_{N-|\sigma^i|}(1)\) either \(T\) is left-full below \((\rho^i\cdot\psi(\tau), \sigma^i\cdot\nu)\) or \((\rho^i\cdot\psi(\tau), \sigma^i\cdot\nu)\notin T\). In other word, \(\psi\) is such that
    \begin{align}
        \big(\rho^i\cdot\psi(\tau), \sigma^i\cdot\nu\big)\in T\implies
        T\text{ is left-full below }\big(\rho^i\cdot\psi(\tau), \sigma^i\cdot\nu\big)\label{eq:x}
    \end{align}
    
    Now for each \(i<2\), define the set
    \[B_i\coloneqq\big\{(\rho^i\cdot\tau, \sigma^i\cdot\nu)\in T: (\tau, \nu)\in \X\text{ and }(\rho^i\cdot\psi(\tau), \sigma^i\cdot\nu)\in T\big\}\]
    And consider \(S \subseteq T\), the downward-closure of $B_0 \cup B_1$. 
    We claim that \(S\) is a cross-subtree of \(T\) for which \((\rho^i, \sigma^i)_{i<2}\) is a condition-tuple. 

    Indeed, fix some~$i  <2$ and let $\tau \in 3^{N - |\rho^i|}$.
    Since $T$ is left-full below $(\rho^i, \sigma^i)$, by \Cref{left-full-ext}, there is some~$\nu \in \X_{N-|\sigma^i|}(1)$ such that $(\rho^i \cdot \psi(\tau), \sigma^i \cdot \nu) \in T$.
    Thus, by \ref{eq:x}, $T$ is left-full below $(\rho^i \cdot \psi(\tau), \sigma^i \cdot \nu)$
    and by definition of~$B_i$, $(\rho^i \cdot \tau, \sigma^i \cdot \nu) \in B_i \subseteq S$.
    Thus $S$ is left-full below~$(\rho^i, \sigma^i)$.

    By \Cref{konig-sufficient}, there is a tuple $(\widehat{\rho}^i, \widehat{\sigma}^i)_{i < 2} \in A \cap S^2$
    such that \(\forall i<2, (\widehat\rho^i, \widehat\sigma^i)\succcurlyeq(\rho^i, \sigma^i)\), and \((\widehat\rho^0, \widehat\rho^1)\) is completely incompatible over \((\rho^0, \rho^1)\).
    Since $A$ is closed under suffix, we can suppose without loss of generality that $|\widehat\rho^i| = |\widehat\sigma^i| = N$, hence that $(\widehat\rho^i, \widehat\sigma^i) \in B_i$. Let \(\tau^i\in \X_{N-|\rho^i|}(0)\) and \(\nu^i\in\X_{N-|\sigma^i|}(1)\) be such that
    $\widehat\rho^i = \rho^i \cdot \tau^i$ and $\widehat\sigma^i = \sigma^i \cdot \nu^i$.
    By definition of $B_i$, $(\rho^i\cdot\psi(\tau^i), \sigma^i\cdot\nu^i)\in T$,
    thus by \Cref{eq:x}, $T$ is left-full below $(\rho^i\cdot\psi(\tau^i), \sigma^i\cdot\nu^i)$.
    Thus, $(\rho^i\cdot\psi(\tau^i), \sigma^i\cdot\nu^i)_{i < 2}$ is a condition-tuple for~$[T]$.
    Moreover, since $\psi$ preserves incompatibility over \((\rho^0, \rho^1)\), $(\rho^i\cdot\psi(\tau^i), \sigma^i\cdot\nu^i)_{i < 2}$ is an extension of \((\rho^i, \sigma^i)_{i<2}\).
\end{proof}

\subsection{Cross-constraint cone avoidance basis theorem}

We now prove our first cross-constraint basis theorem which requires some sort of first-jump control, using the combinatorics developed in \Cref{sect:combi-lemmas}. This basis theorem was first proven by Liu~\cite[Lemma 4.5]{liu2023coding} using a different argument. Our new proof follows more closely the standard proof of the cone avoidance basis theorem for $\Pi^0_1$ classes.

\begin{theorem}[Cross-constraint cone avoidance, {Liu~\cite[Lemma 4.5]{liu2023coding}}]\label{cone-avoidance-basis}
    Let \(C\) be a non-computable set.
    Any left-full computable instance \(T\) of \(\CC\) has a solution \((X^i, Y^i)_{i<2}\) such that \((X^0, Y^0)\oplus(X^1, Y^1)\not\geqslant_T C\).
\end{theorem}
\begin{proof}
    To prove the theorem, we use forcing with conditions of the form \(\left((\rho^i, \sigma^i)_{i<2}, U, B\right)\), where
    \begin{itemize}
        \item \(U\) is a \(B\)-computable cross-subtree of~$T$
        \item \((\rho^i, \sigma^i)_{i<2}\) is a condition-tuple for \([U]\)
        \item \(B\subseteq\NN\) and \(B\not\geqslant_T C\)
    \end{itemize}
    
    We will satisfy the following requirements for each $e \in \NN$:
    \[\R_e : \Phi_e^{(X^0, Y^0)\oplus(X^1, Y^1)}\neq C\]

    A condition \(\left((\rho^i, \sigma^i)_{i<2}, U, B\right)\) \emph{forces} \(\R_\Psi\), if \(\R_\Psi\) holds for all \((X^i, Y^i)\in [U]\) extending \((\rho^i, \sigma^i)\) for each \(i<2\).

    \begin{lemma}\label[lemma]{cone-avoidance-basis-forcing-req}
    For every condition \(c \coloneqq \left((\rho^i, \sigma^i)_{i<2}, U, B\right)\) and every~$e \in \NN$, there is an extension of $c$ that forces \(\R_e\).
    \end{lemma}
    \begin{proof}
    For all \(x \in \NN, v < 2\), define 
    \[A_{x,v}\coloneqq\{(\widehat\rho^i, \widehat\sigma^i)_{i<2}\in\X_{<\NN}^2:\Phi_e^{(\widehat\rho^0, \widehat\sigma^0)\oplus (\widehat\rho^1, \widehat\sigma^1)}(x)\downarrow = v\}\]
    The set $A_{x,v}$ is upward-closed, and~\(\Sigma^0_1\) uniformly in \(x, v\). Consider the following $\Sigma^0_1(B)$ set:
    $$
    Q = \{ (x, v) : A_{x,v} \mbox{ is } U\mbox{-sufficient over } (\widehat\rho^i, \widehat\sigma^i)_{i<2} \}
    $$

    \emph{Case 1.} $(x, C(x)) \not\in Q$ for some~$x \in \NN$. Let \(\L\) be the $\Pi^0_1(B)$ class of cross-trees \(S\) witnessing that \(A_{x, C(x)}\) is not \(U\)-sufficient over \((\rho^i, \sigma^i)_{i<2}\). By the cone avoidance basis theorem, there is some cross-tree \(S\) such that \(S \oplus B\not\geqslant_T C\).
    The condition \(d\coloneqq\big((\rho^i, \sigma^i)_{i<2}, S, S\oplus B\big)\) is the extension we are looking for. Indeed, it forces \(\R_e\), because for \((X^i, Y^i)_{i<2}\in[d]\), if \(\Psi^{(X^0, Y^0)\oplus(X^1, Y^1)}\) is total, then it is different from~$C$ on input~$x$.
    \bigskip
    
    \emph{Case 2.} $(x, 1-C(x)) \in Q$ for some~$x \in \NN$. Unfolding the definition, $A_{x, 1-C(x)}$ is $U$-sufficient \((\rho^i, \sigma^i)_{i<2}\), so by \Cref{sufficient-imply-extendible}, there is a condition-tuple  \((\widehat\rho^i, \widehat\sigma^i)_{i<2}\in A_{x, 1-C(x)}\) extending \((\rho^i, \sigma^i)_{i<2}\). Thus the
    condition \(\big((\widehat\rho^i, \widehat\sigma^i)_{i<2}, U, B\big)\) is the desired extension, as it forces \(\R_e\).
    \bigskip

    \emph{Case 3.} Neither Case 1 nor Case 2 holds. Then $Q$ is the $\Sigma^0_1(B)$ graph of the characteristic function of~$C$, so $C \leq_T B$. Contradiction. 

    
    
    \end{proof}

    We are now ready to prove \Cref{cone-avoidance-basis}. Let $c_0 = \left((\rho^i, \sigma^i)_{i<2}, U, B\right)$ be a condition that excludes a maximal number of components and let $\F$ be a sufficiently generic filter for this notion of forcing, containing~$c_0$. For~$i < 2$, let $X^i = \bigcup \{ \rho^i : \left((\rho^i, \sigma^i)_{i<2}, U, B\right) \in \F \}$ and $Y^i = \bigcup \{ \sigma^i : \left((\rho^i, \sigma^i)_{i<2}, U, B\right) \in \F \}$. By \Cref{left-full-ext}, $X^i$ and $Y^i$ are both infinite sequences. By \Cref{lem:delta-basis-excludes}, for every~$s < r$, $Y^0_s \cap Y^1_s$ is infinite. By \Cref{cone-avoidance-basis-forcing-req}, \((X^0, Y^0)\oplus(X^1, Y^1)\not\geqslant_T C\). This completes the proof of \Cref{cone-avoidance-basis}.
\end{proof}

\begin{corollary}
    For any non-computable set \(C\subseteq\NN\), there is a cross-constraint ideal that does not contain \(C\).
\end{corollary}
\begin{proof}
    We construct a sequence of sets \(Z_0\leqslant_T Z_1\leqslant_T\ldots\) such that for any integer \(n=\langle k, e\rangle\), \(Z_n \not \geq_T C\), and if \(\Phi_e^{Z_k}\) is an instance of \(\CC\), then \(Z_{n+1}\) computes a solution.

    Define \(Z_0\coloneqq\emptyset\). Suppose \(Z_n\) has been defined and let $n = \langle k, e \rangle$. If \(\Phi_e^{Z_k}\) is not a left-full cross-tree, then \(Z_{n+1}\coloneqq Z_n\). Otherwise, by \Cref{cone-avoidance-basis} relativized to \(Z_n\), there is a pair of paths \(P_0\) and \(P_1\), such that \(P_0\oplus P_1\oplus Z_n\not\geqslant_T C\). In which case \(Z_{n+1}\coloneqq P_0\oplus P_1\oplus Z_n\).

    By construction, the class \(\M\coloneqq\{X\in 2^\NN:\exists n, X\leqslant_T Z_n\}\) is a cross-constraint ideal containing only sets avoiding the cone above \(C\), in particular \(C\) is not in the ideal.
\end{proof}

\subsection{\texorpdfstring{Cross-constraint preservation of non-$\Sigma^0_1$ definitions}{Cross-constraint preservation of non-Sigma01 definitions}}

We now prove a second cross-constraint basis theorem, about preservation of non-$\Sigma^0_1$ definitions.
This basis theorem for $\Pi^0_1$ classes was first proven by Wang~\cite{wang2016definability}, and implies the cone avoidance basis theorem in a straightforward way.
Later, Downey et al.~\cite{downey2022relationships} actually proved that the two basis theorems are equivalent, as any problem satisfying any of them, satisfies both. Thus, the following theorem is a (non-trivial) consequence of \Cref{cone-avoidance-basis}. However, we give a direct proof of it, to get familiar with the combinatorics of the cross-constraint problem.

\begin{theorem}[Cross-constraint preservation of non-$\Sigma^0_1$ definitions]\label{sigma01-basis}
    Let \(C\) be a non-$\Sigma^0_1$ set.
    Any computable instance \(T\) of \(\CC\), has a solution \((X^i, Y^i)_{i<2}\) such that \(C\) is not \(\Sigma^0_1\) relative to \((X^0, Y^0)\oplus(X^1, Y^1)\).
\end{theorem}
\begin{proof}
    To prove the theorem, we use forcing with conditions of the form \(((\rho^i, \sigma^i)_{i<2}, U, B)\) where
    \begin{itemize}
        \item \(U\) is a \(B\)-computable cross-subtree of~$T$
        \item \((\rho^i, \sigma^i)_{i<2}\) is a condition-tuple for \([U]\)
        \item \(B\subseteq\NN\) is such that \(C\) is not \(\Sigma^0_1(B)\)
    \end{itemize}

    We want to satisfy the following requirements for every Turing index~$e$:
    \[\R_e : W_e^{(X^0, Y^0)\oplus(X^1, Y^1)}\neq C\]

    \begin{lemma}\label[lemma]{sigma01-basis-forcing-req}
    For every condition \(c \coloneqq \left((\rho^i, \sigma^i)_{i<2}, U, B\right)\) and every~$e \in \NN$, there is an extension of~$c$ forcing \(\R_e\).
    \end{lemma}
    \begin{proof}
    Given some \(x \in \NN\), consider the set 
    \[A_{x}\coloneqq\{(\widehat\rho^i, \widehat\sigma^i)_{i<2}\in\X_{<\NN}^2 : x\in W^{(\widehat\rho^0, \widehat\sigma^0)\oplus (\widehat\rho^1, \widehat\sigma^1)}_e\}\]

    Here again, the set $A_x$ is upward-closed and $\Sigma^0_1$ uniformly in~$x$.
    Let 
    $$Q\coloneqq\{x : A_{e,x}\text{ is }U\text{-sufficient over } (\rho^i, \sigma^i)_{i<2}\}$$
    The set $Q$ is \(\Sigma^0_1(B)\), thus \(Q\neq C\).
    This leads to two cases.
    \bigskip
    
    \emph{Case 1.} There is \(x\in C\smallsetminus Q\). Let $\L$ be the class of all cross-trees~$S\subseteq U$ which witness that $A_x$ is not \(U\)-sufficient over \((\rho^i, \sigma^i)_{i<2}\). It is non-empty by hypothesis, and since~$A_x$ is \(\Sigma^0_1\), then $\L$ is $\Pi^0_1(B)$. 
    Now since \(\WKL\) admits preservation of non-$\Sigma^0_1$ definitions (see \cite[Theorem 3.6]{wang2016definability}), there is a cross-tree \(S\in\L\) such that \(C\) is not \(\Sigma^0_1(S\oplus B)\). The condition \(d \coloneqq\big((\rho^i, \sigma^i)_{i<2}, S, S\oplus B\big)\) is the extension we are looking for, since \(x\in C\) but \(d\) forces that \(x\notin W^{(X^0,Y^0)\oplus (X^1,Y^1)}_e\).
    \bigskip
    
    \emph{Case 2.} There is \(x\in Q\smallsetminus C\).
    Unfolding the definition, \(A_x\) is \(U\)-sufficient over \((\rho^i, \sigma^i)_{i<2}\), so by \Cref{sufficient-imply-extendible}, there is a condition-tuple \((\widehat\rho^i, \widehat\sigma^i)_{i<2}\in A_x\) extending \((\rho^i, \sigma^i)_{i<2}\). The condition \(\big((\widehat\rho^i, \widehat\sigma^i)_{i<2}, U, B\big)\) is the desired extension, as it forces \(x \in W^{(X^0, Y^0)\oplus (X^1, Y^1)}_e\) for some~$x \not \in C$.
    \end{proof}

    We are now ready to prove \Cref{sigma01-basis}. Let $c_0 = \left((\rho^i, \sigma^i)_{i<2}, U, B\right)$ be a condition that excludes a maximal number of components and let $\F$ be a sufficiently generic filter for this notion of forcing, containing~$c_0$. For~$i < 2$, let $X^i = \bigcup \{ \rho^i : \left((\rho^i, \sigma^i)_{i<2}, U, B\right) \in \F \}$ and $Y^i = \bigcup \{ \sigma^i : \left((\rho^i, \sigma^i)_{i<2}, U, B\right) \in \F \}$. By \Cref{left-full-ext}, $X^i$ and $Y^i$ are both infinite sequences. By \Cref{lem:delta-basis-excludes}, for every~$s < r$, $Y^0_s \cap Y^1_s$ is infinite. By \Cref{sigma01-basis-forcing-req}, $C$ is not \(\Sigma^0_1((X^0, Y^0)\oplus(X^1, Y^1))\). This completes the proof of \Cref{sigma01-basis}.
\end{proof}

\begin{corollary}
    For any non-\(\Sigma^0_1\) set \(C\subseteq\NN\). There is a cross-constraint ideal such that \(C\) is not \(\Sigma^0_1\) relative to any element of the ideal.
\end{corollary}
\begin{proof}
    We construct a sequence of sets \(Z_0\leqslant_T Z_1\leqslant_T\ldots\) such that for any integer \(n=\langle k, e\rangle\), \(C\) is not \(\Sigma^0_1(Z_n)\), and if \(\Phi_e^{Z_k}\) is an instance of \(\CC\), then \(Z_{n+1}\) computes a solution.

    Define \(Z_0\coloneqq\emptyset\). Suppose \(Z_n\) has been defined, and let $n = \langle k, e\rangle$. If \(\Phi_e^{Z_k}\) is not a left-full cross-tree, then \(Z_{n+1}\coloneqq Z_n\). Otherwise, by \Cref{sigma01-basis} relativized to \(Z_n\), there is a pair of paths \(P_0\) and \(P_1\), such that \(C\) is not \(\Sigma^0_1\) relative to \(P_0\oplus P_1\oplus Z_n\). In which case \(Z_{n+1}\coloneqq P_0\oplus P_1\oplus Z_n\).

    By construction, the class \(\M\coloneqq\{X\in 2^\NN:\exists n, X\leqslant_T Z_n\}\) is a cross-constraint ideal such that \(C\) is not \(\Sigma^0_1\) relative to any element in it.
\end{proof}

\begin{corollary}[Cross-constraint cone avoidance]
    Let \(C\) be a non-computable set.
    Any left-full computable instance \(T\) of \(\CC\), has a solution \((X^i, Y^i)_{i<2}\) such that \((X^0, Y^0)\oplus(X^1, Y^1)\not\geqslant_T C\).
\end{corollary}
\begin{proof}
Suppose $C$ is non-computable. Then either~$C$ or $\overline{C}$ is not $\Sigma^0_1$.
By \Cref{sigma01-basis},  there is a solution \((X^i, Y^i)_{i<2}\in[T]^2\) such that either $C$ or $\overline{C}$ is not \(\Sigma^0_1\) relative to \((X^0, Y^0)\oplus(X^1, Y^1))\). In particular, \((X^0, Y^0)\oplus(X^1, Y^1)\not\geqslant_T C\).
\end{proof}

\subsection{Cross-constraint low basis theorem}

The low basis theorem for $\Pi^0_1$ classes is one of the most famous theorems in computability theory. We prove its counterpart for the cross-constraint problem. However, contrary to the case of $\Pi^0_1$ classes, where the theorem can be strengthened to obtain superlow sets, it does not seem to be the case for cross-constraint problems. 

\begin{theorem}[Cross-constraint low basis]\label{low-basis}
    Any left-full computable instance \(T\) of \(\CC\), has a solution \((X^i, Y^i)_{i<2}\) such that \((X^0, Y^0)\oplus(X^1, Y^1)\) is low.
\end{theorem}
\begin{proof}
    To prove the theorem, we use forcing with conditions of the form
    \[\left((\rho^i, \sigma^i)_{i<2}, U, B\right)\]
    where
    \begin{itemize}
        \item \(U\) is a $B$-computable cross-subtree of $T$
        \item \((\rho^i, \sigma^i)_{i<2}\) is a condition-tuple for $[U]$
        \item \(B\) is a set of low degree
    \end{itemize}
    An \emph{index} for a condition $\left((\rho^i, \sigma^i)_{i<2}, U, B\right)$ is a tuple $\left((\rho^i, \sigma^i)_{i<2}, a, b\right)$ such that $\Phi^B_a = U$ and $\Phi^{\emptyset'}_b = B'$. An index is therefore a finite representation of a condition.
    We say that a condition \(c \coloneqq \left((\rho^i, \sigma^i)_{i<2}, U, B\right)\) \emph{decides the jump on~$e$} 
    if either \(\Phi^{(\rho^0, \sigma^0)\oplus (\rho^1, \sigma^1)}_e(e)\downarrow\) holds or $c$ forces \(\Phi^{(X^0, Y^0)\oplus(X^1, Y^1)}_e(e)\uparrow\).
    
    \begin{lemma}\label[lemma]{lem:low-basis-deciding-jump}
    For every condition \(c \coloneqq \left((\rho^i, \sigma^i)_{i<2}, U, B\right)\) and every~$e \in \NN$,
    there is an extension $d$ of~$c$ deciding the jump on~$e$. Moreover, an index for~$d$ can be found $\emptyset'$-uniformly in~$e$ and an index for~$c$.
    \end{lemma}
    \begin{proof}
    Consider the following $\Sigma^0_1$ set 
    \[A_e\coloneqq\{(\widehat\rho^i, \widehat\sigma^i)_{i<2}\in\X_{<\NN}^2 : \Phi^{(\widehat{\rho}^0, \widehat{\sigma}^0)\oplus (\widehat{\rho}^1, \widehat{\sigma}^1)}_e(e)\downarrow\}\]
    
    \emph{Case 1. \(A_e\) is not \(U\)-sufficient over \((\rho^i, \sigma^i)_{i<2}\).}
    Let $\L$ be the class of all cross-trees~$S\subseteq T$ which witness that $A_e$ is not \(U\)-sufficient over \((\rho^i, \sigma^i)_{i<2}\). Since~$A_e$ is \(\Sigma^0_1\), then $\L$ is $\Pi^0_1(B)$. 
    By the uniform low basis theorem relative to~$B$ (see \cite[Theorem 4.1]{hirschfeldt2008strength}), there is some \(S \in \L\) such that \((S\oplus B)' \leqslant_T \emptyset'\). Moreover, a lowness index of $S\oplus B$ (an integer~$a$ such that $\Phi^{\emptyset'}_a = (S \oplus B)'$) can be $\emptyset'$-computed from an index of \(\L\).
    The condition \(d\coloneqq((\rho^i, \sigma^i)_{i<2}, S, S\oplus B)\) is the extension we are looking for. Indeed, \(\Phi^{(X^0, Y^0)\oplus(X^1, Y^1)}_e(e)\uparrow\) holds for any \((X^i, Y^i)_{i<2}\in [d]\).
    \bigskip

    \emph{Case 2. \(A_e\) is \(U\)-sufficient over \((\rho^i, \sigma^i)_{i<2}\).} By \Cref{sufficient-imply-extendible}, there is a condition-tuple \((\widehat\rho^i, \widehat\sigma^i)_{i<2}\in A_e\) for \([U]\) which extends \((\rho^i, \sigma^i)_{i<2}\).
    Thus, the condition \(\big((\widehat\rho^i, \widehat\sigma^i)_{i<2}, T, B\big)\) is the desired extension, since \(\Phi_e^{(\widehat\rho^0, \widehat\sigma^0)\oplus (\widehat\rho^1, \widehat\sigma^1)}(e)\downarrow\).
    
    
    Finally, note that \(\emptyset'\) can decide whether or not it is in the first or second case, since \(A_e\) is \(\Sigma^0_1\), and so ``\(A_e\) is \(T\)-sufficient over \((\rho^i, \sigma^i)_{i<2}\)" also is. Hence, each extension can be uniformly computed from \(\emptyset'\).
    \end{proof}

    We are now ready to prove \Cref{low-basis}. We build a uniformly $\emptyset'$-computable descending sequence of conditions $c_0 \geq c_1 \geq \dots$ such that for every~$n$, letting $c_n = \left((\rho^i_n, \sigma^i_n)_{i<2}, U_n, B_n\right)$
    \begin{itemize}
        \item $c_{n+1}$ decides the jump on~$n$ ;
        \item $|\rho^i_n| \geq n$ ; $|\sigma^i_{n,s}| \geq n$ for every~$s < r$;
        \item $(\sigma^0_{n+1,s}, \sigma^1_{n+1,s})$ is not completely incompatible over $(\sigma^0_{n,s}, \sigma^1_{n,s})$.
    \end{itemize}
    Let $c_0 = \left((\rho^i, \sigma^i)_{i<2}, U, B\right)$ be a condition which excludes a maximal number of components. Note that $c_0$ does not need to be found in~$\emptyset'$, since it is a one-time guess.
    Assuming $c_n$ has been defined, by \Cref{lem:low-basis-deciding-jump}, there is an extension $c^1_n \leq c_n$
    deciding the jump on~$n$. By \Cref{left-full-ext}, there is an extension $c^2_n \leq c^1_n$ satisfying the second item, and by \Cref{lem:delta-basis-excludes}, there is an extension $c_{n+1} \leq c^2_n$ satisfying the third item.
    Moreover, indices for each of these extensions can be found $\emptyset'$-computably uniformly in~$n$.
    This completes the proof of \Cref{low-basis}.
\end{proof}

\begin{corollary}
    There is a cross-constraint ideal that contains only low sets.
\end{corollary}
\begin{proof}
    We construct a sequence of sets \(Z_0\leqslant_T Z_1\leqslant_T\ldots\) such that for any integer \(n=\langle k, e\rangle\), \(Z_n\) is low, and if \(\Phi_e^{Z_k}\) is an instance of \(\CC\), then \(Z_{n+1}\) computes a solution.

    Define \(Z_0\coloneqq\emptyset\). Suppose \(Z_n\) has been defined, and let~$n = \langle k, e \rangle$. If \(\Phi_e^{Z_k}\) is not a left-full cross-tree, then \(Z_{n+1}\coloneqq Z_n\). Otherwise, by \Cref{low-basis} relativized to \(Z_n\), there is a pair of paths \(P_0\) and \(P_1\), such that \((P_0\oplus P_1\oplus Z_n)'\leqslant_T Z_n'\). In which case \(Z_{n+1}\coloneqq P_0\oplus P_1\oplus Z_n\).

    By construction, the class \(\M\coloneqq\{X\in 2^\NN:\exists n, X\leqslant_T Z_n\}\) is a cross-constraint ideal containing only low sets.
\end{proof}



\subsection{Yet some other antibasis theorem}

As noted by Liu~\cite[section 4.6]{liu2023coding}, if two 2-colorings are completely incompatible over a pair, then they are Turing equivalent. Then, instead of requiring that two 2-colorings are infinitely compatible, one could strengthen the requirement and ask them not to be Turing equivalent. He therefore asked the following question.

\begin{question}[{Liu~\cite[Question 4.25]{liu2023coding}}]
Given two incomputable oracles $D_0 \not \geq_T D_1$, a non-empty $\Pi^0_1$ class $\P \subseteq 2^\NN$,
does there exist an $X \in \P$ such that $X \not \geq_T D_0$ and $D_0 \oplus X \not \geq_T D_1$?
\end{question}

There is a negative answer thanks to a theorem of Day and Reimann~\cite{day2014independence}:

\begin{theorem}[{Day and Reimann~\cite[Corollary 2.1]{day2014independence}}]
Suppose that $X$ has PA degree and $C$ is a c.e. set. Then either $C \oplus X \geq_T \emptyset'$, or $X \geq_T C$.
\end{theorem}

Thus, letting $D_0$ be a incomplete non-computable c.e. set, $D_1 = \emptyset'$ and $\P$ be a non-empty $\Pi^0_1$ class of only PA degrees, there is no $X \in \P$ satisfying simultaneously $X \not \geq_T D_0$ and $D_0 \oplus X \not \geq_T D_1$.

\section{\texorpdfstring{\(\Gamma\)-hyperimmunity}{Gamma-hyperimmunity}}\label[section]{sect:gamma-hyperimmunity}

\Cref{param-1} can be instantiated by considering the cross-constraint ideal of all arithmetic sets and letting $f$ be a Cohen arithmetically-generic 3-coloring of the integers. However, to obtain a separation of $\mathsf{D}^2_3$ from $(\mathsf{D}^2_2)^*$ over computable reduction, one needs to build a $\Delta^0_2$ such coloring~$f$. Liu~\cite{liu2023coding} designed a new invariance property called preservation of $\Gamma$-hyperimmunity, which is satisfied by weak K\"onig's lemma and the cross-constraint problem, but not by $\mathsf{D}^2_2$. We re-define this notion and prove that it is preserved by $\COH$, which is a new result, enabling to separate $\mathsf{D}^2_3$ from $(\RT^2_2)^*$ over computable reduction.

The notion of $\Gamma$-hyperimmunity might seem ad-hoc at first sight, but becomes clearer when looking at the proof of preservation of $\Gamma$-hyperimmunity of the cross-constraint problem (Liu~\cite[Lemma 4.2]{liu2023coding}) for which it was specifically designed.

Recall that  an instance \(f\) of \(\RT^1_k\) is \emph{hyperimmune} relative to \(D\subseteq\NN\) if for every \(D\)-computable sequence $g_0, g_1, \dots$ of partial functions from $\NN$ to $k$ with finite support, such that $\min \dom g_n > n$, there is some \(N\in\NN\) such that $g_n$ is compatible with~$f$. $\Gamma$-hyperimmunity is a strengthening of this notion of hyperimmunity, between $\emptyset$-hyperimmunity and $\emptyset'$-hyperimmunity, by considering computable approximations of sequences of partial functions with finite mind changes, for a very specific family of approximations. They seem to be closely related to the Ershov hierarchy~\cite{ershov1968certain}.

The following series of definitions (\Cref{def:one-step-variation,def:computation-path,def:gamma-spaces,def:interpretation-computation,def:over,def:gamma-m-approximation} yield the notion of $\Gamma$-approximation, from which $\Gamma$-hyperimmunity is derived.

\begin{definition}\label[definition]{def:one-step-variation}
    A tree \(T_1\subseteq \NN^{<\NN}\) is a \emph{one-step variation} of \(T_0\subseteq \NN^{<\NN}\) if there is a node \(\xi\in T_0\) and
    a non-empty finite set \(F\subseteq\NN\) such that
    \begin{itemize}
        \item either \(\xi\in\ell(T_0)\) and \(T_1=T_0\cup\xi\cdot F\) 
        \item or \(\xi\in T_0-\ell(T_0)\), \(T_1=(T_0-[\xi]^\prec)\cup\xi\cdot F\) and \(F\subsetneq \{n \in \NN : \xi\cdot n \in T_0\}\)
    \end{itemize}
\end{definition}

In other words, a one-step variation of a tree consists of either extending a leaf with finitely many immediate children, or backtracking by removing a node, and making all its siblings leaves again. This is a non-reflexive relation. The notion of one-step variation can be reminiscent of the Hydra game~\cite{kirby1982accessible}.

\begin{definition}\label[definition]{def:computation-path}
    Fix a partial order \((W, \preccurlyeq)\) which is a tree of root \(\zeta\). A \emph{computation path} on \((W, \preccurlyeq)\) is a finite sequence \((T_0, \varphi_0), (T_1, \varphi_1), \ldots, (T_{u-1}, \varphi_{u-1})\) where, for all \(j<u\), \(T_j\subseteq\NN^{<\NN}\) is a finite tree such that 
    \begin{itemize}
        \item \(T_0=\{\varepsilon\}\)
        \item \(j\in\NN\), \(T_{j+1}\) is a one-step variation of \(T_j\)
    \end{itemize}
    And, for all \(j<u\), \(\varphi_j:T_j\to W\) is a function such that
    \begin{itemize}
        \item \(\varphi_j(\varepsilon)=\zeta\)
        \item \(\varphi_j\) is non-decreasing 
        \item \(\varphi_{j+1}\) and \(\varphi_j\) are compatible, \ie \(\varphi_{j+1}=\varphi_j\) on the domain \(T_{j+1}\cap T_j\)
    \end{itemize}
\end{definition}

We shall consider exclusively well-founded trees \((W, \preccurlyeq)\), in which case any computation path is of finite length. The notion of computation path can be used as an operator to define an infinite hierarchy of well-founded trees.

\begin{definition}[Gamma spaces]\label[definition]{def:gamma-spaces}
    By induction, we define the partial orders \((\Gamma_m, \preccurlyeq_m)\) that are trees of root \(\zeta_m\)
    \begin{itemize}
        \item \(\Gamma_0\coloneqq\{f:\NN\to 3: \dom f\text{ is finite}\}\), forming a tree of depth 1, with the empty map \(\zeta_0\) as its root, and every other element as an immediate child of the root.
        \item \(\Gamma_{m+1}\) is the set of computation paths on \((\Gamma_m, \preccurlyeq_m)\), its root \(\zeta_{m+1}\) is the computation path \((\{\varepsilon\}, \varepsilon\mapsto\zeta_m)\), and \(\preccurlyeq_{m+1}\) is the prefix relation on sequences.
    \end{itemize}
\end{definition}

Intuitively, the root of $\Gamma_1$ is the nowhere-defined function, the immediate children are finite sets of functions with finite support, and the sub-branches consist of removing elements from this finite set. 

\begin{lemma}[{Liu~\cite[Lemma 4.12]{liu2023coding}}]\label{gamma-well-founded}
    For all \(m\in\NN\), the tree \(\Gamma_m\) is well-founded.
\end{lemma}
\begin{proof}
    By induction on \(m\), we show that the structure of \(\Gamma_m\) is the following: there are infinitely many nodes of height 1, and for each node of height 1, the subtree above it is finite. In particular, it is a well-founded tree.
    
    The tree \(\Gamma_0\) satisfies the statement, by construction. Now suppose \(\Gamma_m\) has the above structure for some \(m\in\NN\). A computation path on \(\Gamma_m\) will first select finitely many nodes, say \(\xi_0, \ldots, \xi_{n-1}\), of height 1.
    
    Since the function in a computation path is an embedding into \(\Gamma_m\), it means that the subtree above \(\xi_i\) (for any \(i<n\)) is finite. Furthermore, the ``either" case can only be applied finitely many times to a node. Indeed, the ``or" case can turn an inner node back into a leaf, but it can only do so finitely many times, because it strictly decreases the cardinal of the set of direct successors of the node it is applied to. Hence, there can only be finitely many computation paths.
\end{proof}

Due to the inductive nature of the definition of $\Gamma$-spaces, its elements are relatively abstract. One must think of a computation path in $\Gamma_m$-space as a finite set of partial functions from $\NN$ to $k$ with finite support.

\begin{definition}\label[definition]{def:interpretation-computation}
The \emph{interpretation} $\inter{\gamma}$ of a computation path $\gamma\in\Gamma_m$ is a finite non-empty subset of \(\Gamma_0\) defined inductively as follows: 
\begin{itemize}
    \item if $m = 0$, $\inter{\gamma} = \{\gamma\}$
    \item if $m > 0$ and \(\gamma\coloneqq((T_0, \varphi_0), \ldots, (T_{u-1}, \varphi_{u-1}))\), then $\inter{\gamma} = \bigcup_{\zeta \in \ell(T_{u-1})} \inter{\zeta}$
\end{itemize}
\end{definition}


\begin{definition}\label[definition]{def:over}
For \(n\in\NN\), a set~$F \subseteq \Gamma_0$ is \emph{over} \(n\) if for every~$g \in F$, $\dom g \subseteq ]n, \infty[$.
By extension, for \(m, n\in\NN\), we say \(\gamma\in\Gamma_m\) is \emph{over \(n\)} if \(\inter\gamma\) is over \(n\).
\end{definition}

    

Based on the interpretation of a computation path~$\gamma \in \Gamma_m$ as a subset of~$\Gamma_0$,
a $\Gamma_m$-approximation is a $\Delta^0_2$-approximation of a list of finite subsets of~$\Gamma_0$,
for which the finite mind changes is ensured by an increasing sequence in a well-founded tree, or equivalently as a  decreasing sequence of ordinals.

\begin{definition}\label[definition]{def:gamma-m-approximation}
    For \(m\in\NN\), a \emph{\(\Gamma_m\)-approximation} is a function \(f:\NN\times\NN\to\Gamma_m\) (for some \(m\)) if for all \(n\in\NN\)
    \begin{itemize}
        \item \(f(n, 0)=\zeta_m\)
        \item \(f(n, -)\) is non-decreasing 
        \item \(\forall s\in\NN, f(n,s)\text{ is over }n\) 
    \end{itemize}
    We also define its interpretation as 
    \begin{align*}
        \inter{f}:\NN\times\NN &\to\mathfrak{P}_\text{fin}(\Gamma_0)\\
        n, s &\mapsto\inter{f(n, s)}
    \end{align*}
\end{definition}

It is often useful to consider enumerations of Turing functionals which furthermore satisfy some decidable syntactic constraint. We shall see that one can computably list all the Turing functionals behaving as $\Gamma_m$-approximations.

\begin{definition}
    For any \(m\in\NN\), a Turing functional \(\Psi\) is a \emph{\(\Gamma_m\)-functional} if and only if it is total and \(\Psi^X\) is a \(\Gamma_m\)-approximation for every oracle \(X\).
\end{definition}
\begin{lemma}\label{gamma-functional}
    For every \(m\in\NN\) and every Turing functional \(\Xi\), there is a \(\Gamma_m\)-functional \(\Psi\) such that, for any oracle \(X\), if \(\Xi^X\) is a \(\Gamma_m\)-approximation, then \(\Psi^X\) has the same limit function.
\end{lemma}
\begin{proof}
The functional \(\Psi\) proceeds as follows. Fix an oracle \(X\) and some \(n\in\NN\). Define \(\Psi^X(n, 0)\coloneqq\zeta_m\). Now to define \(\Psi^X(n, t)\) for \(t>0\), consider \(s<t\) the biggest integer (if it exists) such that \(\Xi^X(n, s)[t]\downarrow\) and such that \(\Xi^X(n, s)[t]\succ\Psi^X(n, t-1)\). If \(s\) exists then \(\Psi^X(n, t)\coloneqq\Xi^X(n, s)[t]\). Otherwise, \(\Psi^X(n, t)\coloneqq\Psi^X(n, t-1)\).
\end{proof}


As mentioned, a $\Gamma$-approximation should be thought of as a $\Delta^0_2$-approximation of sequences of finite subsets of~$\Gamma_0$ of a special shape, and such that the minimum of the support of its elements is unbounded. 
We are now ready to define the notion of $\Gamma$-hyperimmunity. Given a $\Gamma$-approximation, it is sufficient to be compatible with any partial function of the limit to satisfy it.

\begin{definition}
   A function \(f:\NN\to 3\) (potentially partial) is \emph{compatible with} a set $F \subseteq \Gamma_0$ if $f$ extends some element in~$F$, \ie \(\exists g\in F, \dom f \supseteq \dom g\text{ and }f\uh_{\dom g} = g\). By extension, for \(m\in\NN\), a partial 3-coloring \(f\) is \emph{compatible with} \(\gamma\in\Gamma_m\), if it is compatible with \(\inter\gamma\).
\end{definition}

\begin{definition}\label[definition]{def:gamma-hyperimmunity}
    A 3-coloring \(f\) is \emph{\(\Gamma\)-hyperimmune} relative to \(D\subseteq\NN\) if, for every \(m\in\NN\) and for every \(D\)-computable \(\Gamma_m\)-approximation \(g\), there is an \(n\in\NN\) such that \(f\) is compatible with \(\lim_s g(n, s)\).
\end{definition}

The following lemma shows that $\Gamma$-hyperimmunity is a stronger notion than hyperimmunity. It simply comes from the fact that a computable list of elements of~$\Gamma_0$ with unbounded minimum support is a particular case of $\Gamma_0$-approximation.

\begin{lemma}\label{gamma-impl-hyp}
    If a 3-coloring \(f\) is \(\Gamma\)-hyperimmune relative to \(D\subseteq\NN\), then it is also hyperimmune relative to \(D\).
\end{lemma}
\begin{proof}
    Let \((F_{n, 0}, F_{n, 1}, F_{n, 2})_{n\in\NN}\) be a \(D\)-computable sequence of mutually disjoint finite sets \((F_{n, 0}, F_{n, 1}, F_{n, 2})_{n\in\NN}\) such that \(\min\bigcup_{j<3} F_{n, j}>n\).

    For any \(n\in\NN\), consider we define a 3-coloring with finite support 
    \[\gamma_n:x\mapsto\begin{dcases}
            j&\text{if }x\in F_{n, j}\\
            \uparrow&\text{otherwise}
        \end{dcases}\]
    It is well defined since the finite sets \(F_{n, j}\) are mutually disjoint. And consider the function
    \begin{align*}
        g :\ &\NN^2\to\Gamma_0\\
            &n, s\mapsto
            \begin{dcases}
                \emptyset&\text{if }s=0\\
                \gamma_n&\text{otherwise}
            \end{dcases}
    \end{align*}
    It is a \(D\)-computable \(\Gamma_0\)-approximation, thus by \(\Gamma\)-hyperimmunity of \(f\) there is \(N\in\NN\) such that \(f\) is compatible with \(\lim_s g(N, s)\), \ie \(f\) extends \(\gamma_N\), \ie \(\forall j<3, F_{N, j}\subseteq f^{-1}(j)\). 
\end{proof}

Liu~\cite{liu2023coding} proved the following lemma, which we reprove with more details.

\begin{proposition}[{Liu~\cite[Lemma 4.17]{liu2023coding}}]\label[proposition]{exists-gamma-hyp-coloring}
    There is a \(\Delta^0_2\) coloring which is \(\Gamma\)-hyperimmune.
\end{proposition}
\begin{proof}
    First, it is possible to computably list all computable \(\Gamma_m\)-approximations, where \(m\) is any integer. Indeed, given a computable partial order \((W, \preccurlyeq)\), the set of its computation paths is uniformly computable. This is because they are finite sequences composed of finite trees, so all the constraints listed in the definition can be computed. From there, the set of \(\Gamma_m\)-spaces is uniformly computable in~\(m\). Hence, we fix an enumeration \((\Phi_{e,m})_{e\in\NN}\) of all the computable \(\Gamma_m\)-approximations, for any~\(m\).
 
    We wish to construct \(f:\NN\to 3\) such that,
    for any \(e, m\in\NN\), the following requirement is satisfied.
    \[\R_{e,m}\coloneqq f\text{ is compatible with }\Phi_{e,m}\]

    Suppose we have so far constructed \(\rho\in 3^{<\NN}\), we now consider the \(\Gamma_m\)-approximation \(\Phi_{e,m}\), and we are going to use \(\emptyset'\) to find an integer \(s\) such that \(\Phi_{e,m}(|\rho|, s)\) is the limit value for \(\Phi_{e,m}(|\rho|, -)\). Then either \(\inter{\Phi_{e,m}(|\rho|, s)}=\emptyset\), in which case \(\R_{e,m}\) is satisfied. Or \(\inter{\Phi_{e,m}(|\rho|, s)}\neq\emptyset\), in which case \(\R_{e,m}\) is satisfied for \(\rho\cup\tau\) where \(\tau\in\Phi_{e,m}(|\rho|, s)\). Because, by definition of a \(\Gamma_m\)-approximation, \(\tau\) is over \(|\rho|\), \ie \(\min\dom\tau > |\rho|\)
    
    To find \(s\in\NN\), build a sequence of integers, starting with \(s_0\coloneqq 0\). If \(s_k\) is the latest integer we have defined, then use \(\emptyset'\) to know whether \(\exists y>s_k, \Phi_{e,m}(|\rho|, y)\succ_m \Phi_{e,m}(|\rho|, s_k)\).
    If the answer is yes then \(s_{k+1}\) is defined by such a witness, otherwise we stop defining the sequence and \(s\coloneqq s_k\).

    The above procedure must end at some point, because \(\Phi_{e,m}(|\rho|, s_0)\prec_m \Phi_{e,m}(|\rho|, s_1)\prec_m\ldots\) is a strictly increasing sequence in the space \(\Gamma_m\), which is well-founded (by \Cref{gamma-well-founded}). And so \(\Phi_{e,m}(|\rho|, s)\) is the limit value we were looking for, because the case that defines \(s\) ensures that \(\forall y>s, \Phi_{e,m}(|\rho|, y)=\Phi_{e,m}(|\rho|, s)\).
\end{proof}

The following two theorems by Liu state that $\WKL$ and $\CC$ both admit preservation of $\Gamma$-hyperimmunity.
Their proofs are quite technical but essentially follow from the above combinatorics of the cross-constraint problem. We therefore do not include them.

\begin{theorem}[{Liu~\cite[Lemma 4.18]{liu2023coding}}]
    Let \(f:\NN\to 3\) be \(\Gamma\)-hyperimmune. For every non-empty $\Pi^0_1$ class~$\P \subseteq 2^\NN$, there is a member \(X \in \P\) such that \(f\) is \(\Gamma\)-hyperimmune relative to~\(X\).
\end{theorem}

\begin{theorem}[{Liu~\cite[Lemma 4.2]{liu2023coding}}]\label[theorem]{thm:cc-preservation-gamma}
    Let \(f:\NN\to 3\) be \(\Gamma\)-hyperimmune. For every computable instance of \(\CC\), there is a solution \(X\) such that \(f\) is \(\Gamma\)-hyperimmune relative to~\(X\).
\end{theorem}

\subsection{\texorpdfstring{Preservation of \(\Gamma\)-hyperimmunity for~$\COH$}{Preservation of Gamma-hyperimmunity for COH}}

We now prove that $\COH$ admits preservation of $\Gamma$-hyperimmunity, which is a new result, and yields in particular that $\mathsf{D}^2_3$ is not computably reducible to $(\RT^2_2)^*$ (\Cref{separation2}).

\begin{theorem}\label[theorem]{thm:coh-preserves-gamma}
Let~$g \in 3^\NN$ be a \(\Gamma\)-hyperimmune function and $R_0, R_1, \dots$ be a uniformly computable sequence of sets.
Then there is an $\vec R$-cohesive set~$G$ such that $g$ is \(\Gamma\)-hyperimmune relative to~$G$.
\end{theorem}
\begin{proof}
    We proceed by forcing, using Mathias conditions \((F, X)\) such that \(X\) is computable.

    For a $\Gamma_m$-functional \(\Psi\), define the requirement \(\R_{\Psi, m}\coloneqq\) there is \(n\in\NN\) such that \(g\) is compatible with \(\lim_s \Psi^G(n, s)\).
    
    \begin{lemma}\label{coh-preservation-forcing-step}
        For each condition \((F, X)\), \(m\in\NN\), and $\Gamma_m$-functional \(\Psi\), there is an extension forcing \(\R_{\Psi, m}\)
    \end{lemma}
    \begin{proof}
        We define a computable \(\Gamma_m\)-approximation~$f : \NN^2 \to \Gamma_m$ as follows.
        First, for every $n$, $f(n,0)\coloneqq \zeta_m$.
        Suppose that at step $s$, we have defined $f(n,s)$ for every~$n$.
        Then for each \(n\), if there is some~$F' \subseteq X$ with~$\max F' < s$ and some~$t\leqslant s$ such that $\Psi^{F\cup F'}(n, t)\downarrow$ and $\Psi^{F  \cup F'}(n, t)\succ f(n,s)$, then let $f(n, s+1)\coloneqq \Psi^{F \cup F'}(n, t)$. Otherwise, let $f(n, s+1)\coloneqq f(n,s)$. Then go to the next stage.
        
        By construction, $f$ is a \(\Gamma_m\)-approximation. Since $g$ is \(\Gamma\)-hyperimmune, there is \(n\in\NN\) such that \(g\) is compatible with \(\lim_s f(n, s)\). Now by definition of \(f\), there is a finite (possibly empty) \(F'\subseteq X\) and \(t\in\NN\) such that \(\lim_s f(n, s)=\Psi^{F\cup F'}(n, t)\). We claim that $(F\cup F', X-\{0, \dots, \max F'\})$ forces $\R_{\Psi, m}$.
        Indeed, by construction of \(f\) and since we considered the limit, we have that, for every \(F''\subseteq X -\{0, \dots, \max F'\} \) and cofinitely many \(t'\in\NN\), \(\Psi^{F\cup F' \cup F''}(n, t')=\Psi^{F\cup F'}(n, t)\).
    \end{proof}

    Let~$\mathcal{F}$ be a sufficiently generic filter for computable Mathias forcing and let $G = \bigcup_{(F, X) \in \mathcal{F}} F$. By genericity, $G$ is $\vec{R}$-cohesive, since given a condition $(F, X)$ and a computable set~$R_x$, either~$(F, X \cap R_x)$ or $(F, X \cap \overline{R}_x)$ is a valid extension. By \Cref{coh-preservation-forcing-step}, for every~$m$ and $\Gamma_m$-functional~$\Gamma$, $g$ diagonalizes against $\Gamma^G$. Thus, by \Cref{gamma-functional}, \(g\) is \(\Gamma\)-hyperimmune relative to \(G\).
\end{proof}

\begin{corollary}\label{exists-ideal-coh}
    For every $\Gamma$-hyperimmune function $f : \NN \to 3$, there exists a cross-constraint ideal $\M \models \COH$
    such that $f$ is $\Gamma$-hyperimmune relative to every element of~$\M$.
\end{corollary}
\begin{proof}
We construct a sequence of sets \(Z_0\leqslant_T Z_1\leqslant_T\ldots\) such that for any integer \(n=\langle i, k, e\rangle\), 
\begin{itemize}
    \item $f$ is $\Gamma$-hyperimmune relative to~\(Z_n\).
    \item if $n = \langle 0, k, e\rangle$, and if \(\Phi_e^{Z_k}\) is an instance of \(\CC\), then \(Z_{n+1}\) computes a solution.
    \item if $n = \langle 1, k, e\rangle$, and if \(\Phi_e^{Z_k}\) is an instance of \(\COH\), then \(Z_{n+1}\) computes a solution.
\end{itemize}

Define \(Z_0\coloneqq\emptyset\). Suppose \(Z_n\) has been defined.
If~$n = \langle 0, k, e \rangle$ and \(\Phi_e^{Z_k}\) is not a left-full cross-tree, then \(Z_{n+1}\coloneqq Z_n\). Otherwise, by \Cref{thm:cc-preservation-gamma} relativized to \(Z_n\), there is a pair of paths \(P_0\) and \(P_1\), such that $f$ is $\Gamma$-hyperimmune relative to~\(P_0\oplus P_1\oplus Z_n\). In which case \(Z_{n+1}\coloneqq P_0\oplus P_1\oplus Z_n\).
If~$n = \langle 1, k, e \rangle$ and \(\Phi_e^{Z_k}\) is not a countable sequence of sets, then \(Z_{n+1}\coloneqq Z_n\). Otherwise, by \Cref{thm:coh-preserves-gamma} relativized to \(Z_n\), there is a \(\Phi_e^{Z_k}\)-cohesive set $C$, such that $f$ is $\Gamma$-hyperimmune relative to~\(C \oplus Z_n\). In which case \(Z_{n+1}\coloneqq C \oplus Z_n\).

By construction, the class \(\M\coloneqq\{X\in 2^\NN:\exists n, X\leqslant_T Z_n\}\) is a cross-constraint ideal such that $f$ is $\Gamma$-hyperimmune relative to every element of~$\M$.
\end{proof}

\subsection{Reducibility results}\label[section]{sect:reducibility}

\begin{definition}
    A problem \(\Psf\) is \emph{strongly omnisciently computably reducible} to a problem \(\Qsf\), noted \(\Psf\leqslant_{soc}\Qsf\), if for every $\Psf$-instance~$X$, there exists a $\Qsf$-instance~$\widehat{X}$ such that, for every $\Qsf$-solution~$\widehat{Y}$ of~$\widehat{X}$, $\widehat{Y}$ computes a $\Psf$-solution to~$X$.
\end{definition}

Contrary to computable reduction, no effectiveness is imposed on the complexity of the $\Qsf$-instance $\widehat{X}$ to solve the $\Psf$-instance~$X$. Moreover, the solution to the $\Psf$-instance has to be computable from the solution~$\widehat{Y}$ to the $\Qsf$-instance~$\widehat{X}$, without the help of~$X$.

\begin{theorem}[{Liu~\cite[Theorem 2.1]{liu2023coding}}]\label[theorem]{thm:rt13-soc-rt12}
    \(\RT^1_3\not\leqslant_{\text{soc}}(\RT^1_2)^{*}\)
\end{theorem}
\begin{proof}
    Let \(\M\) be a countable cross-constraint ideal (such an ideal exists thanks to \Cref{exists-arith-ideal}), and let \(f\in\X(0)\) be hyperimmune relative to any set in \(\M\). For any \(g\in\X(1)\), by \Cref{param-1}, there are sets \(\vec{G}\) witnessing the inequality \(\RT^1_3\not\leqslant_{\text{soc}}(\RT^1_2)^{*}\).
\end{proof}

Liu~\cite[Theorem 4.1]{liu2023coding} proved that $\SRT^2_3 \not \leqslant_c (\SRT^2_2)^{*}$. We strengthen his result by proving that it holds even for non-stable instances of $\RT^2_2$.

\begin{theorem}\label{separation2}
    \(\SRT^2_3 \not\leqslant_c (\RT^2_2)^{*}\)
\end{theorem}
\begin{proof}
    By \cref{exists-gamma-hyp-coloring} there exists a \(\Delta^0_2\) coloring \(f:\NN\to 3\) which is $\Gamma$-hyperimmune.
    Using Shoenfield's limit lemma, there is a stable computable function~$h : [\NN]^2 \to 3$ such that for every~$x$, $\lim_y h(x, y) = f(x)$. Consider $h$ as a computable instance of~$\SRT^2_3$. Fix any $r$-tuple of computable colorings \(g_0, \ldots, g_{r-1}:[\NN]^2\to 2\) for some \(r\in\NN\). It suffices to show the existence of $g_i$-homogeneous sets~$H_i$ for every~$i < r$ such that $\bigoplus_{j<r} H_j$ does not compute any infinite $h$-homogeneous set.
    
    By \Cref{exists-ideal-coh}, there is a countable cross-constraint ideal \(\M\models\COH\) for which \(f\) is \(\Gamma\)-hyperimmune relative to any set in \(\M\). In particular, since $g_0, \dots, g_{r-1}$ are computable, they belong to~$\M$. Moreover, by \Cref{gamma-impl-hyp}, $f$ is hyperimmune relative to every element of~$\M$. By \Cref{param-2}, there exists  $g_i$-homogeneous sets~$H_i$ for every~$i < r$, such that $\bigoplus_{j<r} H_j$ does not compute any infinite $f$-homogeneous set. Since any $h$-homogeneous set is $f$-homogeneous, then $\bigoplus_{j<r} H_j$ does not compute any infinite $h$-homogeneous set.
\end{proof}

\begin{definition}
    A set \(H\subseteq\NN\) is \emph{pre-homogeneous} for a coloring \(f:[\NN]^{n+1}\to k\) if \(\forall \sigma\in [H]^n, \forall z\gt y\gt \max\sigma, f(\sigma, y)=f(\sigma, z)\)
\end{definition}

Jockusch~\cite[Lemma 5.4]{jockusch1972ramsey} proved that for every computable coloring \(f:[\NN]^{n+1}\to k\), every PA degree relative to~$\emptyset'$ computes an infinite pre-homogeneous set for~$f$. Moreover, for~$n \geq 2$, Hirschfeldt and Jocksuch~\cite[Theorem 2.1]{hirschfeldt2016notions} proved a reversal.

\begin{theorem}
For every~$n \geq 2$,
    \(\SRT^n_3 \not\leqslant_c (\RT^n_2)^{*}\)
\end{theorem}
\begin{proof}
We prove by induction on~$n \geq 2$ that for every set~$P$, there exists a $\Delta^0_n(P)$ coloring $f : \NN \to 3$
such that for every $r \geq 1$ every $r$-tuple of $P$-computable colorings $g_0, \dots, g_{r-1} : [\NN]^n \to 2$,
there are infinite $\vec{g}$-homogeneous sets \(G_0, \dots, G_{r-1}\) such that \(\vec G \oplus P\) does not compute any infinite $f$-homogeneous set.

The case \(n=2\) corresponds to a relativized form of \Cref{separation2}. Now suppose the hypothesis holds for some \(n\in\NN\). Fix some set~$P$, and let \(Q\gg P'\) be such that \(Q'\leqslant_T P''\) (which exists by relativization of the low basis theorem (Jockusch and Soare~\cite{jockusch197classes})).

    By induction hypothesis relativized to \(Q\), there exists a \(\Delta^0_n(Q)\) (\ie \(\Delta^0_{n+1}(P)\))  coloring \(f:\NN\to 3\) such that for every $r \geq 1$ every $r$-tuple of $Q$-computable colorings $g_0, \dots, g_{r-1} : [\NN]^n \to 2$, there are infinite $\vec{g}$-homogeneous sets \(G_0, \dots, G_{r-1}\) such that \(\vec G \oplus Q\) does not compute any infinite $f$-homogeneous set.

    Now, consider an $r$-tuple of $P$-computable colorings $h_0, \dots, h_{r-1} : [\NN]^{n+1} \to 2$.
    By Jockusch~\cite[Lemma 5.4]{jockusch1972ramsey}, $Q$ computes infinite sets \(C_0, \dots, C_{r-1}\subseteq\NN\) pre-homogeneous for \(\widetilde h_0, \dots, \widetilde h_{r-1}\). For~$s < r$, let $g_s : [\NN]^n \to 2$
    be the $Q$-computable coloring defined by $g_s(i_0, \dots, i_{n-1}) = h_s(x^s_{i_0}, \dots, x^s_{i_{n-1}}, y)$, where $C_s = \{x^s_0 < x^s_1 < \dots \}$ and $y \in C_s$ is any element bigger than $x^s_{i_{n-1}}$.

    By choice of~$f$, there are infinite $\vec g$-homogeneous sets \(G_0, \dots, G_{r-1}\) such that \(\vec G \oplus Q\) does not compute any infinite $f$-homogeneous set. In particular, letting $H_s = \{x^s_i : i \in G_s\}$, $H_s$ is $h_s$-homogeneous, and since $\vec H \oplus P \leq_T \vec G \oplus Q$, then 
    $\vec H \oplus P$ does not compute any infinite $f$-homogeneous set.
    This completes our induction.

    Finally, by Shoenfield's limit lemma, for every~$n \geq 2$, there exists a stable computable coloring $\widehat f : [\NN]^n \to 3$ such that any infinite $\widehat f$-homogeneous set is $f$-homogeneous, where $f : \NN \to 3$ is the function witnessed by the inductive proof.


\end{proof}

\bibliographystyle{plain}
\bibliography{biblio.bib}

\end{document}